%% file: compactness_shrinker_R_v3.tex
\documentclass[11pt]{amsart}
\usepackage{amssymb}
\usepackage{amsfonts}
\usepackage{amsmath}
\usepackage{graphicx}
\usepackage{color}
\usepackage{amsmath}
\usepackage{stmaryrd}
\usepackage{epsfig,color}
\usepackage{blindtext}
\usepackage{enumerate}
\usepackage{hyperref}
\usepackage{url}
\usepackage{bbm}
\usepackage{filecontents}
\usepackage{nicefrac,mathtools}
\DeclareGraphicsExtensions{.pdf,.jpeg,.png}
\usepackage{epstopdf}
\usepackage{cancel} %for editing
\usepackage[normalem]{ulem} %for editing
\usepackage{verbatim}		%for comment
\usepackage{soul} % use this for strike out

\usepackage{color}
\usepackage[msc-links, lite]{amsrefs}
\usepackage{geometry}
\newtheorem{theorem}{Theorem}[section]

\newtheorem{proposition}[theorem]{Proposition}
\newtheorem{lemma}[theorem]{Lemma}
\newtheorem{corollary}[theorem]{Corollary}

\newtheorem{claim}[]{Claim}

\theoremstyle{definition}

\theoremstyle{remark}
\newtheorem{remark}[theorem]{Remark}

\def\R{\mathbb R}

\def\n{\mathbf n}

\def\area{\mathrm{Area}}

\def\dist{\mathrm{dist}}
\newcommand{\mc}{\mathcal}
\newcommand{\mb}{\mathbb}
\newcommand{\wti}{\widetilde}

\numberwithin{equation}{section}

\title{Compactness of self-shrinkers in $\R^3$ with fixed genus}
\date{\today}

\author{Ao Sun}
\address{Department of Mathematics, Massachusetts Institute of Technology, 77 Massachusetts Avenue,
Cambridge, MA 02139, USA}
\email{aosun@mit.edu}

\author{Zhichao Wang}
\address{Max-Planck Institute for Mathematics, Vivatsgasse 7, 
53111 Bonn, Germany}
\email{wangzhichaonk@gmail.com}

\begin{document}
\begin{abstract}
We prove the compactness of self-shrinkers in $\mathbb R^3$ with bounded entropy and fixed genus. As a corollary, we show that numbers of ends of such surfaces are uniformly bounded by the entropy and genus.
\end{abstract}

\maketitle
\section{Introduction}
A hypersurface $\Sigma\subset\R^{n+1}$ is called a \emph{self-shrinker} if it satisfies the equation
\begin{equation}
H=\frac{1}{2}\langle x,\mathbf n\rangle,
\end{equation}
where $H$ is the mean curvature of the hypersurface, $x$ is the position vector on the hypersurface, and $\mathbf{n}$ is the unit normal vector of $\Sigma$ at $x$.

Self-shrinkers are the homothetic solitons of mean curvature flow. That is, suppose $\Sigma$ is a self-shrinker, then $\{\sqrt{-t}\Sigma\}_{t\in(-\infty,0)}$ is a solution to the mean curvature flow. Moreover, by using Huisken's monotonicity formula in \cite{Hui90}, Ilmanen \cite{Ilm95} and White \cite{Whi94} proved that self-shrinkers are models for singularities of mean curvature flows. 

The simplest self-shrinkers in $\mathbb R^3$ are the planes, the spheres, and the cylinders. Brendle \cite{Bre16} proved that these are the only possible self-shrinkers with genus $0$. Angenent \cite{Ang92} constructed a genus $1$ self-shrinker which is known to be the Angenent donut. For higher genus, Kapouleas-Kleene-M{\o}ller \cite{KKM18} and Nyugen \cite{Nguyen} constructed a number of examples of non-compact self-shrinkers by a gluing method.

In this paper, we prove the compactness of self-shrinkers in $\mathbb R^3$ with fixed genus. 
\begin{theorem}\label{main thm}
Let $\{\Sigma_i\}\subseteq\R^3$ be a sequence of embedded self-shrinkers with genus $g$ and uniformly bounded entropy. Then there exist a subsequence (still denoted by $\{\Sigma_i\}$) and a self-shrinker $\Sigma$ with genus $g$ such that $\Sigma_i$ converges to $\Sigma$ in $C_{loc}^\infty(\mathbb{R}^3)$.
\end{theorem}

Here entropy is defined to be the Gaussian area of self-shrinkers (see (\ref{def:entropy}) for the precise definition).  

Given $\Lambda>0$ and a non-negative integer $g$, we define $\mathcal M_\mathcal S(\Lambda,g)$ to be the space of self-shrinkers in $\mathbb R^3$ with genus $g$ and entropy no more than $\Lambda$. For the sake of brevity, set 
\begin{equation}\label{def of M space}
\mathcal M_\mathcal S(g)=\bigcup_{\Lambda>0}\mathcal M_\mathcal S(\Lambda,g).
\end{equation}
Then our main theorem is equivalent to say that $\mathcal M_\mathcal S(\Lambda,g)$ is compact under $C_{loc}^\infty(\mathbb{R}^3)$ topology.

The smooth compactness theorem of $\bigcup_{k=0}^g\mathcal M_\mathcal S(\Lambda,k)$ was first developed by Colding-Minicozzi in \cite{CM12_2}, see also some further generalizations in \citelist{\cite{CZ13}\cite{CMZ15}}. However, it is not known that whether the topological type of the self-shrinkers would change in the limit. For example, when passing to limit, a ``hole" on a sequence of self-shrinkers may move further and further, and vanishes at infinity when passing to limit. Our theorem rules out this possibility, and improves the understanding of the structure of self-shrinkers with finite genus. 

As a result, a direct corollary of our main theorem is

\begin{corollary}\label{cor:uniform decomposition}
Given $\Lambda>0$, $g>0$, there exists $R=R(\Lambda,g)>0$ such that for any $\Sigma\in \mathcal M_\mathcal S(\Lambda, g)$, $\Sigma\setminus B_R(0)$ consists of finite union of topological annuli and topological disks.
\end{corollary}

In other words, we prove that the ``holes" on a self-shrinker can not stay far away from the origin of $\mathbb R^3$.

\vspace{1.5em}

The classification of self-shrinkers' ends plays an important role in our paper. Such a result has been firstly studied in \cite{Wang14} by L. Wang, who used a Carleman type technique to prove that two self-shrinkers' ends which are asymptotic to the same cone must identically coincide with each other. Later in \cite{Wang16_2}, L. Wang proved that two self-shrinkers' ends which are asymptotic to the same cylinder with certain decay rate must identically coincide with each other. As a result, the asymptotic behavior of noncompact self-shrinkers at infinity is quite clear. Recently, L. Wang classified the self-shrinkers' ends with finite topology in \cite{Wang16}, which is the most important structure theory of self-shrinkers in $\mathbb{R}^3$.

It may be interesting to compare the structure theory of self-shrinkers with the structure theory of minimal surfaces in $\R^3$, cf. Colding-Minicozzi \citelist{\cite{CM-minimal1}\cite{CM-minimal2}\cite{CM-minimal3}\cite{CM-minimal4}\cite{CM-minimal5}}, Meeks-Rosenberg \cite{MeeksRosenberg05}. The structure of minimal surfaces with finite topology and finite total curvature has been well studied, and the ends of minimal surfaces look like planes or helicoids at infinity. This is very different from the case of self-shrinkers. 

The key step in \cite{Wang16} is to prove the multiplicity of a tangent flow of a self-shrinker equals to $1$. Then the classification theorem follows from Brakke's Regularity Theorem (see \citelist{\cite{Bra78}\cite{Whi02}}). Note that the argument in \cite{Wang16} works for all $\Sigma\in\mathcal M_\mathcal S(\Lambda,g)$. Consequently, $\Sigma$ can only have finitely many ends (see Lemma \ref{lem:finite ends}). Our theorem in this paper implies that the number of ends are bounded by a constant depending only on $\Lambda$ and $g$.

\begin{corollary}\label{cor:uniform bound ends}
There exists a constant $C=C(\Lambda ,g)$ such that for all $\Sigma\in\mathcal M_\mathcal S(\Lambda, g)$, the number of ends of $\Sigma$ is no more than $C$.
\end{corollary}

\vspace{1.5em}
Another application of Theorem \ref{main thm} is the existence of entropy minimizer in $\mathcal M_\mathcal S(g)$. It follows from the definition of entropy that the plane $\mathbb R^2$ minimizes the entropy among all surfaces in $\R^3$. Among all closed self-shrinkers in $\R^3$, the round sphere is the only one achieves the least entropy for all dimensions, which was proven by Colding-Ilmanen-Minicozzi-White in \cite{CIMW}. After that, Bernstein-Wang \cite{BW17} confirmed that the cylinders have the third least entropy among all self-shrinkers in $\mathbb R^3$. As a corollary of our main theorem, we show the existence of an entropy minimizer with fixed genus $g$. 
\begin{corollary}
There exists an entropy minimizer in $\mathcal M_\mathcal S(g)$ whenever $\mathcal M_\mathcal S(g)\neq \emptyset$.
\end{corollary}

We remark that, Bernstein-Wang \cite{BW16} proved that the entropy of closed hypersurfaces is minimized by round spheres among all closed surfaces in $\mathbb R^{n+1}$ with $n\leq 6$. Later Zhu generalized this result to all dimensions in \cite{Zhu16}. In $\mathbb R^3$, Ketover-Zhou \cite{KZ} gave an alternative proof using min-max theory.

This kind of minimizing problem is of great interest in the study of geometry. For example, Simon \cite{Sim93} proved that for a fixed genus number $g$ (see also \cite{BK}), there is a genus $g$ Willmore surface in $\mathbb{R}^3$ minimizes Willmore energy among all closed surfaces with genus $g$. In contrast, for a sequence of minimal surfaces or constant mean curvature surfaces, the topology may change when passing to a limit. Therefore in general it is hard to find an extremal surface with fixed genus.

\vspace{1.5em}
\paragraph{\it Outline of the proof} 
By the compactness of $\cup_{k\leq g}\mathcal M_\mathcal S(\Lambda,k)$ proved by Colding-Minicozzi in \cite{CM12_2}, $\Sigma_i$ converges to $\Sigma$ smoothly in arbitrarily large balls $B_R$. This fact implies that $\Sigma_i\cap B_R$ is diffeomorphic to $\Sigma\cap B_R$ when $i$ sufficiently large. Thus we only need to analyze $\Sigma_i$ outside a large ball. By Wang's result \cite{Wang16}*{Theorem 1.1}, $\Sigma\setminus B_R(0)$ can be decomposed to finitely many annuli. Moreover, each connected component is either asymptotic to a regular cone or a round cylinder. We will call them {\em conical} or {\em cylindrical} ends respectively. Our purpose is to prove that the connected components of $\Sigma_i\setminus B_R$ are topological annuli or topological disks. There are two cases.

\vspace{0.7em}
Firstly, consider that $\Gamma$ is a connected component of $\Sigma\setminus B_R(0)$ which is conical. $\Sigma_i\setminus B_R(0)$ can also be decomposed into several connected components. We consider the one which is close to $\Gamma$, and denote it by $\Gamma_i$. By the monotonicity formulas of $F$-functional near the asymptotic cone, we can prove that $\Gamma_i$ completely lies in a tubular neighborhood of $\Gamma$ (see Lemma \ref{lem:end curve hausdorff converges}). If $|x|^2$ has no critical points in such a neighborhood, then by Morse theory we know that $\Gamma_i$ has to be a topological annulus. So we only need to prove that $|x|^2$ has no critical points on $\Gamma_i$.

To prove that $|x|^2$ has no critical points on $\Gamma_i$, we argue it by contradiction. Suppose $p_i\in \Gamma_i$ is a critical point of the function $|x|^2$. Then $\Sigma_i$ locally smoothly converging to $\Sigma$ implies that $|p_i|\rightarrow\infty$. We can blow up these ends at these critical points, and the limit would be a translating soliton which is not a plane. Moreover it is a Bowl soliton (see Theorem \ref{Theorem: blow up shrinker to get translator}). However, the monotonicity formulas show that the entropy of the translator should be bounded by the $F$-functional of $\Sigma$ at $\partial\Gamma$ (see Proposition \ref{prop:entropy upper bound}), which is closed to $1$ since $\Gamma$ is asymptotic to a regular cone. Recall that the entropy of a Bowl soliton equals to $\lambda_1$ (the entropy of $S^1$; see \cite{Gua16}) which is larger than $1$. Thus we get a contradiction. So $\Gamma_i$ is diffeomorphic to an annulus.

\vspace{0.7em}
Next consider that $\Gamma$ is a cylindrical end with direction $y$. The simplest case is that $\Gamma_i$ has an end which is also asymptotic to a cylinder. Denote the direction of a cylindrical end of $\Gamma_i$ by $y_i$, then $y_i\rightarrow y$ (see Lemma \ref{lem:end curve hausdorff converges}). For any sequence $\rho_i\rightarrow\infty$, we claim that $\Sigma_i-\rho_iy_i$ locally smoothly converges to a self-shrinking cylinder. Indeed, by monotonicity formulas, $F$-functional on $\Sigma_i-\rho_iy_i$ is greater or equal than $\lambda_1$ (see Proposition \ref{prop:entropy upper bound}). On the other hand, the limit (denoted by $\nu$) of $\Sigma_i-\rho_i y_i$ has $F$-functional less than or equal to the $F$-functional of $\Sigma-\rho y$, which converges to $\lambda_1$ as $\rho\rightarrow\infty$. Thus $\Sigma_i-\rho_iy_i$ converges to an $F$-stationary varifold with entropy $\lambda_1$. Then the claim follows from Brakke's regularity theorem (see \citelist{\cite{Bra78}\cite{Whi02}} or Theorem \ref{thm:Brakke regularity}) if $\nu$ is supporting on a smooth self-shrinker with multiplicity $1$. 

It suffices to show the smoothness of $\nu$. Though Ilmanen \cite{Ilm95} shows that the singularities of mean curvature flow of surfaces in $\mathbb R^3$ is smooth, we can not follow his argument to prove that  $\nu$ supports on a smooth self-shrinker. The reason is that we need to prove the smoothness of the limits for all $\rho_i\rightarrow\infty$, and then the integral of $H^2$ is hard to control.

To prove $\nu$ is smooth, we show $\nu$ splits off a line and then the claim follows from the entropy bound. Indeed, the limit is also $F$-stationary if we take limits for $\Sigma-(\rho_i+\tau)y_i$ for each $\tau\in\mathbb R$. Thus $\nu-\tau y$ is also $F$-stationary, which implies that $\nu$ splits off a line. The entropy of $\nu$ equals to $\lambda_1$, which implies that $\nu$ is supporting on a self-shrinking cylinder. 

Consequently, $\Sigma_i-\rho_iy_i$ converges to a self-shrinker for any sequence of $\rho_i\rightarrow\infty$. Combining with the fact that $\Sigma_i$ locally smoothly converges to $\Sigma$, we can prove that $\Gamma_i$ is diffeomorphic to $\Gamma$ (see step 4 of Proposition \ref{prop:cylinder end} for details).

When $\Gamma_i$ consists of conical ends or it is compact, the idea is the same but a suitable blow up/down argument is necessary. We refer to Proposition \ref{prop:cone} and \ref{prop:compact} for details.

\vspace{1.5em}
The paper is organized as follows: In Section \ref{section:preliminaries}, we summarize some known results in mean curvature flows and Brakke flows. Then we give an upper bound for entropy after blowing up. Using this, we show that the Bowl soliton is the only non-trivial blowing up limit. In Section \ref{section:main results}, another version (Theorem \ref{thm:shrinker converge}) of Theorem \ref{main thm} will be stated. Using this, the proof of Corollary \ref{cor:uniform decomposition} and \ref{cor:uniform bound ends} follows. The Section \ref{section:cone} is devoted to the first part of Theorem \ref{thm:shrinker converge}, which deals with the case of conical ends. In Section \ref{section:cylinder}, we prove the second part of Theorem \ref{thm:shrinker converge}, which deals with the case of cylindrical ends. In Appendix \ref{section:end classification}, we recall L. Wang's classification result of ends \cite{Wang16}. In Appendix \ref{section:Translating solitons with low entropy}, we recall the classification result of low entropy translating solitons in \cite{Her18}.

\vspace{1em}
\paragraph{\it Acknowledgement}
We want to thank Professor Bill Minicozzi and Professor Lu Wang for the helpful discussion and comments. Z.W. would like to thank Professor Gang Tian and Professor Xin Zhou for constant encouragement.  A.S. wants to thank Kyeongsu Choi for the helpful discussion. We also would like to thank Jonathan Zhu for bringing our attention to the paper of Bernstein-Wang \cite{BW18}.

\section{Preliminaries}\label{section:preliminaries}
In this section, we summarize some known results which are used in this paper. The experts should feel free to skip this section.
\subsection{$F$-functional and monotonicity formulas}\label{subsection:F functional}
Let us recall Huisken's monotonicity formulas. Given $(y,s)\in \mathbb R^{n+1}\times \mathbb{R}$, let $\Phi_{y,s}$ be the backward heat kernel defined on $\mathbb R^{n+1}\times (-\infty,s)$:
\begin{equation} 
\Phi_{y,s}(x,t)=(4\pi(s-t))^{-\frac{n}{2}}e^{-\frac{|x-y|^2}{4(s-t)}}.
\end{equation}

We will omit the subscript if $(y,s)=(0,0)$

For any hypersurface $\Sigma$ in $\mathbb R^{n+1}$ and $(x_0,t_0)\in \mathbb R^{n+1}\times (0,+\infty)$, the {\em $F$-functional} is defined by 
\begin{equation}
F_{x_0,t_0}(\Sigma)=(4\pi t_0)^{-\frac{n}{2}}\int_{\Sigma}e^{-\frac{|x-x_0|^2}{4t_0}} = \int_\Sigma \Phi(x-x_0,-t_0).
\end{equation}

We say $\Sigma$ is a \emph{self-shrinker} in $\mathbb R^{n+1}$ if it satisfies the equation, 
\[ H=\frac{\langle x,\n\rangle }{2}.\]
By monotonicity formulas \cite{CM12_1}*{1.9}, for a self-shrinker $\Sigma$, $V\in\mathbb R^{n+1}$ and $a\in \mathbb R$, $F_{sV,1+as^2}(\Sigma)$ is monotonically decreasing in $s$ whenever $1+as^2>0$.

The {\em entropy} of $\Sigma$ is defined by
\begin{equation}\label{def:entropy}
\lambda(\Sigma):=\sup_{(x_0,t_0)\in\mathbb R^3\times (0,\infty)}F_{x_0,t_0}(\Sigma).
\end{equation}

In this paper, we always denote the entropy of a self-shrinking cylinder in $\mathbb R^3$ by $\lambda_1$.

The entropy is a quantity characterizing all the scales of a hypersurface. This fact is suggested by the following useful proposition (see \citelist{\cite{CM12_1}\cite{CZ13}}).
\begin{proposition}\label{Proposition: entropy and area growth are equivalent}
There is a constant $C$ such that for any embedded hypersurface $\Sigma$ in $\mathbb{R}^{n+1}$,
\begin{equation}
\sup_{x\in\R^{n+1}}\sup_{r>0}\frac{\area (\Sigma\cap B_r(x))}{r^n}\leq C\lambda(\Sigma),
\end{equation}
and
\begin{equation}
\lambda(\Sigma)\leq C\sup_{x\in\R^{n+1}}\sup_{r>0}\frac{\area (\Sigma\cap B_r(x))}{r^n}.
\end{equation}
\end{proposition}

\subsection{Brakkes flows}
We refer to \citelist{\cite{Wang16}*{\S 2}\cite{Ilm95}} for the definition of Brakke flows. The following monotonicity formula is similar to the monotonicity formula of mean curvature flow:
\begin{lemma}[Monotonicity Formulas for Brakke Flows, \cite{Ilm95}*{Lemma 7}]\label{Brakke monotonicity}
Let $\{\mu_t\}_{t\in (a,b)}$ be a Brakke flow satisfying
\[ \sup_{x\in\mathbb R^{n+1}}\sup_{R>0}\frac{\mu_t(B_R(x))}{R^n}\leq C<+\infty.\]
Then for all $(y,s)\in\R^{n+1}\times(a,b)$, and all $a\leq t_1\leq t_2<s,$
\begin{align*}
	&\int_{t_1}^{t_2}\int\Phi_{y,s}(x,t)\Big|H-\frac{(x-y)^\bot}{2(s-t)}\Big|^2d\mu_t(x)dt\\
	\leq&\int\Phi_{y,s}(x,t_1)d\mu_{t_1}(x)- \int\Phi_{y,s}(x,t_2) d\mu_{t_2}(x).
	\end{align*}
\end{lemma}

\begin{remark}
For a Radon measure $\mu$, we can also compute the $F$-functional, which is defined by 
\[F_{x,t}(\mu):=\int (4\pi t)^{-\frac{n}{2}}e^{-\frac{|y-x|^2}{4t}}d\mu(y).\]
Note that if $\mu=\mathcal H^n\lfloor\Sigma$ for a hypersurface $\Sigma$, then $F_{x,t}(\mu)=F_{x,t}(\Sigma)$.
\end{remark}

The following proposition will be used frequently in our argument. This result is well-known and we refer to \cite{Sun18} for a proof.
\begin{proposition}\label{prop:F continuous}
Let $\{\mu_i\}$ be a sequence of Radon measures and 
\[\limsup_{i\rightarrow\infty}\sup_{(x,t)\in\mathbb R^{n+1}\times(0,+\infty)}F_{x,t}(\mu_i)\leq C_0<+\infty.\]
Suppose that $\mu_i\rightarrow\mu$. Then 
$\lim_{i\rightarrow\infty}F_{x,t}(\mu_i)=F_{x,t}(\mu)$.
\end{proposition}

We need the compactness theorem for Brakke flows in the following sections.
\begin{lemma}[Compactness of Brakke flows, \cite{Ilm94}*{Lemma 7.1}]\label{lem:brakke compactness}
Let $\{\{\mu^i_t\}_{t\in[a,b)}\}$ be a sequence of Brakke flows so that for all bounded open $U\subset \mathbb R^{n+1}$,
\begin{equation}\label{eq:bounded ratio}
\limsup_{i\rightarrow\infty}\sup_{t}\mu_t^i(U)\leq C(U) <+\infty.
\end{equation}
Then there is a subsequence $\{i_k\}$ and an integral Brakke flow $\{\mu_t\}_{t\in[a,b)}$, so that $\mu_t^{i_k}\rightarrow\mu_t$ in the sense of Radon measures as $k\rightarrow\infty$.
\end{lemma}

The following theorem is proven by White in \cite{Whi02}, which is a consequence of Brakke's local regularity theorem in \cite{Bra78}. Our statement here is a special case of \cite{Whi02}.
\begin{theorem}[{\cite{Whi02}*{Theorem 7.3}}]\label{thm:Brakke regularity}
Let $\{\Sigma_t^i\}_{t\in(a,b)}$ be a sequence of mean curvature flows in $\mathbb R^{n+1}$ ($n\geq 2$). Suppose that the corresponding Brakke flows $\{\mu_t^i\}$ (where $\mu_t^i:=\mathcal H^n\lfloor\Sigma_t^i$) converge to a Brakke flow $\mu_t$ as Radon measures. If $\mu_t$ is a mean curvature flow measure for $t\in(a,b)$, that is, $\mu_t$ supports on a smooth manifold $\Sigma_t$ with multiplicity $1$, then $\Sigma_t^i$ locally smoothly converge to $\Sigma_t$ for $t\in(a,b)$.
\end{theorem}

\subsection{Blow-up and blow-down arguments}
In this paper, $y\in \mathbb R^3$ also represents a constant vector field. Denote the line in $\mathbb R^3$ containing $y$ by $\mathbb R_y$. For $a,b\in\mathbb R$, set 
\begin{equation}\label{eq:def of I}
I_y(a,b)=\{ty|\, t\in(a,b)\}.
\end{equation}
We will use $\R_y$ to denote $I_y(-\infty,\infty)$.

The following proposition will be used frequently in this paper.
\begin{proposition}\label{prop:entropy upper bound}
Let $\{\Sigma_i\}\subseteq \mathcal M_\mathcal S(\Lambda,g)$ (see (\ref{def of M space})) and $\Sigma_i$ locally smoothly converges to $\Sigma$. Suppose $y_i\in S^2(1)$ and $y_i\rightarrow y$. For any $\rho_i\rightarrow+\infty$ and $a_i>0$ satisfying $\rho_ia_i\rightarrow +\infty$,
define
\[\mu=\lim_{i\rightarrow\infty} \mathcal H^2\lfloor a_i(\Sigma_i-\rho_iy_i).\]
Then for any $(x,t)\in \mathbb R^3\times (0,+\infty)$,
\[F_{x,t}(\mu)=\lim_{i\rightarrow\infty}F_{x,t}(a_i(\Sigma_i-\rho_iy_i))\leq \lim_{\rho\rightarrow\infty}F_{\rho y,1}(\Sigma).\]
\end{proposition}
\begin{proof}
The proof is based on the continuity of $F$-functional (Proposition \ref{prop:F continuous}) and monotonicity formulas (\S \ref{subsection:F functional}). Indeed,
\begin{align*}
F_{x,t}(\mu)=&\lim_{i\rightarrow\infty}F_{x,t}(a_i(\Sigma_i-\rho_iy_i))\\
            =&\lim_{i\rightarrow\infty}F_{0,1}\big(\frac{a_i(\Sigma_i-\rho_iy_i)-x}{\sqrt t}\big)\\
            =&\lim_{i\rightarrow\infty}F_{\rho_i(y_i+x/(\rho_ia_i)),[\rho_i^2(t/a_i^2-1)/\rho_i^2]+1}(\Sigma_i)\\
            \leq & \limsup_{\rho\rightarrow\infty}\limsup_{i\rightarrow\infty}F_{\rho(y_i+x/(\rho_ia_i)),[\rho^2(t/a_i^2-1)/\rho_i^2]+1}(\Sigma_i)\\
            =&\limsup_{\rho\rightarrow\infty} F_{\rho y,1}(\Sigma) \ \ (\text{since } \rho_i\rightarrow\infty \text{ and } a_i\rho_i\rightarrow\infty)\\
            =&\lim_{\rho\rightarrow\infty}F_{\rho y,1}(\Sigma) \ \ (\text{by monotonicity formulas}).
\end{align*}
Thus we complete the proof.
\end{proof}

\begin{theorem}\label{Theorem: blow up shrinker to get translator}
Suppose $\{\Sigma_i\}\subseteq \mathcal M_\mathcal S(\Lambda,g)$ (see (\ref{def of M space})) and $p_i\in \Sigma_i$ such that $|p_i|\to \infty$. Then the blow-up sequence 
\[\widetilde \Sigma_i=\frac{|p_i|}{2}(\Sigma_i- p_i)\]
locally smoothly converges to a Bowl soliton $\widetilde \Sigma$ with direction $-y$, or $\mathbb R_y\times \mathbb R$ up to a subsequence. Here $y$ is the limit point of $p_i/|p_i|$. Moreover if there is a constant $\delta>0$ such that $|A|(p_i)\geq \delta|p_i|$, then $\widetilde \Sigma$ is not a plane.
\end{theorem}

\begin{proof}

By Lemma \ref{lem:A linear bounded}, we conclude that there exists a constant $C$ such that $\Sigma_i$ has uniformly curvature bound 
\[|A^{\Sigma_i}|\leq C(1+|x|),\]
where $A^{\Sigma_i}$ is the second fundamental form of $\Sigma_i$. Then the rescaling sequence $\widetilde \Sigma_i$ has a locally uniformly curvature bound, which implies that there is a subsequence of $\widetilde \Sigma_i$ locally smoothly converging to $\widetilde \Sigma$. Note $\widetilde \Sigma_i$ satisfies the equation
\[{\widetilde  H=\langle\frac{2x}{|p_i|^2},\mathbf{n}\rangle + \langle\frac{p_i}{|p_i|},\mathbf{n}\rangle.}\]
Passing to a limit, then $\widetilde \Sigma$ satisfies the equation 
\[\widetilde  H=\langle y,\mathbf{n}\rangle,\]
where $y=\lim_{i\to \infty}p_i/|p_i|$. Hence the limit $\widetilde \Sigma$ is a translating soliton.

By Proposition \ref{prop:entropy upper bound}, 
\[F_{x,t}(\widetilde \Sigma)\leq \lim_{\rho\rightarrow\infty}F_{\rho y,1}(\Sigma)\leq \lambda_1.\] 
This entropy bound implies that $\widetilde \Sigma$ is a Bowl soliton or $\mathbb R_y\times \mathbb R$ by Theorem \ref{thm:translator low entropy} (see also Choi-Haslhofer-Hershkovits \cite{CHH18}*{Theorem 1.2}).

Now we suppose that there is $\delta>0$ such that $|A|(p_i)\geq \delta|p_i|$. Then after rescaling we conclude that $|\widetilde  A|(0)\geq \delta>0$. Thus the limit cannot be a plane.
\end{proof}

\begin{remark}
The results in this section show a part of the correspondence of self-shrinkers and translating solitons which have been pointed out by Ilmanen in \cite{Ilm94}*{Appendix J}. We believe that there would be a more general and uniform arguments to set up this kind of correspondence.
\end{remark}

\section{Main results}\label{section:main results}

For a subset $U\subset \mb R^3$ and $r\geq 0$, define the related conical end as
\begin{equation}\label{def of C_U}
\mathcal C_{U,r}:=\{tx:x\in U, t\geq r\}.
\end{equation}
We also use $\mc C_U$ to if $r=0$.

In this section, $\{\Sigma_i\}$ is always assumed to be a sequence in $\mathcal M_\mathcal S(\Lambda,g)$. By \cite{CM12_2}, there exists $\Sigma\in\cup_{k\leq g}\mathcal M_\mathcal S(\Lambda ,k)$ such that $\Sigma_i\rightarrow\Sigma$ locally smoothly up to a subsequence (still denoted by $\{\Sigma_i\}$). By L. Wang \cite{Wang16} (see our statement in Theorem \ref{thm:tangent flow multiplicity 1}), 
\[S^2\cap \lim_{\rho\rightarrow\infty}\rho^{-1}\Sigma\]
consists of disjoint simple closed curves and points. Therefore, there exists large $R$ such that $\Sigma\setminus B_R(0)$ consists of self-shrinking ends $\Gamma_1,\Gamma_2,\cdots,\Gamma_k$.

We say that $\mathcal C_{U,R}$ is a {\em conical neighborhood} of $\Gamma_i$ if $\mathcal C_{U,R}\cap\Sigma=\Gamma_i$.

The key ingredient is the following theorem. 
\begin{theorem}\label{thm:shrinker converge}
Let $\{\Sigma_i\}$ be a sequence of self-shrinkers in $\mathcal M_\mathcal S(\Lambda,g)$ (see (\ref{def of M space}) for definition) that locally smoothly converges to $\Sigma$. Suppose that there exist an open set $U\subset S^2(1)$ and $R>0$ such that $\mathcal C_{U,R}\cap\Sigma$ is a self-shrinker end $\Gamma$. Then up to a subsequence, one of the following happens:
\begin{enumerate}
  \item $\Gamma$ is asymptotic to a cone and $\Sigma_i\cap\mathcal C_{U,R}$ is diffeomorphic to $\Gamma$;
  \item $\Gamma$ is asymptotic to a cylinder, and $\Sigma_i\cap\mathcal C_U$ is diffeomorphic to $\Gamma$ or a disk.
\end{enumerate}
\end{theorem}

The proof of Theorem \ref{thm:shrinker converge} will be postponed to \S \ref{section:cone} and \S \ref{section:cylinder}.

In the rest of this section, we use Theorem \ref{thm:shrinker converge} to prove our main results of this paper. We also need Lemma \ref{lem:finite ends} in the Appendix. Let us state it here.

\begin{lemma}[Lemma \ref{lem:finite ends}]
Let $M$ be a self-shrinker with finite entropy and genus. Then $M$ has finite ends.
\end{lemma}

Now we are willing to prove our main results.

\medskip
\begin{proof}[Proof of Theorem \ref{main thm}]
Given a complete surface $M$ in $\mathbb R^3$, denote the number of genus of $M$ by $g(M)$.

Since $\Sigma$ has finite genus, then the classification of ends in \cite{Wang14} and Lemma \ref{lem:finite ends} implies that we can take $R$ large enough such that 
\[g(B_R(0)\cap \Sigma)=g(\Sigma).\]
Note that $\Sigma_i$ locally smoothly converges to $\Sigma$. Hence for large $i$,
\[g(B_R(0)\cap\Sigma_i)=g(B_R(0)\cap\Sigma).\]
Now by Theorem \ref{thm:shrinker converge}, $\Sigma_i\cap \mathcal C_{U,R}$ is diffeomorphic to $S^1\times \mathbb R$ or a disk. Hence 
\[g(\Sigma_i)=g(\Sigma_i\cap B_R(0))=g(\Sigma).\]
Thus we complete the proof.
\end{proof}

\medskip
\begin{proof}[Proof of Corollary \ref{cor:uniform decomposition}]
Assume on the contrary that $\{\Sigma_i\}\subseteq \mathcal M_\mathcal S(\Lambda,g)$ satisfy the following: for any $R>0$, there exists $i>0$ such that one of the connected components of $\Sigma_i\setminus B_R(0)$ is not diffeomorphic to an annulus or a disk.

By Theorem \ref{main thm}, $\Sigma_i$ locally smoothly converge to $\Sigma\in\mathcal M_\mathcal S(\Lambda,g)$. Then there exists $R_0>0$, such that $\Sigma\setminus B_{R_0}(0)$ consists of self-shrinker ends. Then by Theorem \ref{thm:shrinker converge}, each connected component of $\Sigma_i\setminus B_R(0)$ is diffeomorphic to an annulus or a disk, which leads to a contradiction.
\end{proof}

\medskip
Given a complete surface $M$ in $\mathbb R^3$, denote the number of ends of $M$ by $E(M)$. Now we prove the uniform bound of $E$. 
\begin{proof}[Proof of Corollary \ref{cor:uniform bound ends}]
Assume on the contrary that $\{\Sigma_i\}\subseteq \mathcal M_\mathcal S(\Lambda,g)$ and $E(\Sigma_i)\rightarrow\infty$. By Theorem \ref{main thm}, there is a subsequence of $\{\Sigma_i\}$ (still denoted by $\{\Sigma_i\}$) locally smoothly converges to $\Sigma\in\mathcal M_\mathcal S(\Lambda,g)$. Then by Theorem \ref{thm:shrinker converge}, $E(\Sigma)\geq E(\Sigma_i)$ for large $i$. Hence 
\[E(\Sigma)\geq \liminf_{i\rightarrow\infty}E(\Sigma_i)=+\infty,\]
which contradicts Lemma \ref{lem:finite ends}.

\end{proof}

\section{Asymptotic to a cone}\label{section:cone}
Now we are going to prove Theorem \ref{thm:shrinker converge}. 

This section is devoted to the case that $\Sigma\cap \mathcal C_{U,R}$ (see (\ref{def of C_U}) for notations) is asymptotic to a cone. The purpose is to show that $\Sigma_i\cap  \mathcal C_{U,R}$ is also a self-shrinker's end and it is asymptotic to a cone when $i$ is large enough. We first show that $\Sigma_i$ does not intersect $\mathcal C_{\partial U,R}$ for large $i$. 

\begin{lemma}\label{lem:end curve hausdorff converges}
Let $\{\Sigma_i\}$ be a sequence of self-shrinkers in $\mathcal M_\mathcal S(\Lambda,g)$ (see (\ref{def of M space}) for definition) that locally smoothly converges to $\Sigma$. Suppose that $\Sigma\setminus B_R(0)\subset \mc C_{U,R}$ for some real number $R>0$ and an open set $U\subset S^2(1)$. Then $\Sigma_i\setminus B_R(0)\subset\mc C_{U,R}$ for all large $i$. Particularly, $\lim_{\rho\rightarrow\infty }\rho^{-1}\Sigma_i\cap S^2(1)$ converges to a subset of $\lim_{\rho\rightarrow\infty}\rho^{-1}\Sigma\cap S^2(1)$ in the Hausdorff distance.
\end{lemma}
\begin{proof}
Assume on the contrary that there exists $\rho_iy_i\in\Sigma_i$ (here $|y_i|=1$) so that $\rho_i>R$ and $\rho_iy_i\notin \mc C_{U,R}$. Denote by $y$ the limit of $y_i$. Then $y\not\in U$. Since $\Sigma_i$ locally smoothly converges to $\Sigma$, we must have $\rho_i\rightarrow\infty$. 

Take $a_i\rightarrow0$  so that 
\[0<1+a_i\rho_i^2\ll 1.\]
Because $\Sigma_i$ is a smooth surface, we must have
\[F_{\rho_iy_i,1+a_i\rho^2_i}(\Sigma_i)>1-1/i.\]
By the monotonicity formulas in Section \ref{subsection:F functional},
\begin{align*}
F_{\rho y,1}(\Sigma)=
\lim_{i\rightarrow\infty}F_{\rho y_i,1+a_i\rho^2}(\Sigma_i)\geq\liminf_{i\rightarrow\infty}F_{\rho_i y_i,1+a_i\rho_i^2}(\Sigma_i)\geq 1,
\end{align*}
which contradicts that $\Sigma\setminus B_R(0)\subset \mc C_{U,R}$.
\end{proof}

\bigskip
In the rest of this section, we prove the first part of Theorem \ref{thm:shrinker converge}.
\begin{theorem}\label{Theorem: Asymptotic to a cone main theorem}
Let $\Sigma_i,\Sigma, U$ and $\Gamma$ be the same as in Theorem \ref{thm:shrinker converge}. Suppose that there is a curve $\gamma\subseteq S^2(1)$ such that
 \[\lim_{\rho\rightarrow\infty}\rho^{-1}\Sigma\cap U=\gamma,\]
i.e. $\Sigma\cap \mathcal C_{U,R}$ is a conical end for some $R>0$. Then there exist $R>0,i_0>0$ such that $\Sigma_i\cap\mathcal C_{U,R}$ is diffeomorphic to $S^1\times[0,+\infty)$ for $i\geq i_0$.
\end{theorem}

\begin{proof}
Let us consider the function $|x|^2$ defined on $\Gamma_i$. Take $R$ large enough such that $\Sigma\cap \mathcal C_{U,R}$ is diffeomorphic to $S^1\times [0,+\infty)$. 

If $|x|^2$ has no critical points in $\Sigma_i\cap \mathcal C_{U,R}$ for large $i$, $\Sigma_i$ locally smoothly converging to $\Sigma$ implies that the level set $\{|x|=R\}\cap\Sigma_i\cap \overline{\mathcal C_{U,R}}$ is $S^1$ when $i$ is sufficiently large. Then by standard Morse theory we know that $\Sigma_i\cap \mathcal C_{U,R}$ is diffeomorphic to $S^1\times[0,+\infty)$ for sufficiently large $i$.

To complete the proof, it suffices to show that $|x|^2$ has no critical point in $\Sigma_i\cap \mathcal C_{U,R}$ for sufficiently large $i$. We prove it by contradiction. Assume on the contrary that $|x|^2$ has a critical point $p_i\in\Sigma_i\cap \mathcal C_{U,R}$ for large $i$. Since $\Sigma_i$ locally smoothly converges to $\Sigma$, we must have $|p_i|\rightarrow\infty$.

We claim $p_i^\bot=p_i$. Suppose $\{e_1,e_2\}$ is an orthogonal frame around $p_i$ on $\Sigma_i\cap \mathcal C_{U,R}$. Since $p_i$ is a critical point of $|x|^2$,
\[\nabla_{e_j}\langle x,x\rangle(p_i)=0,\ \ \ \ j=1,2.\]
Note that
\[\nabla_{e_j}\langle x,x\rangle(p_i) = 2\langle e_i,p_i\rangle,\]
which implies that $p_i^\top=0$, hence $p_i^\bot= p_i$.

Now set $y_i=p_i/|p_i|$ and $\rho_i=|p_i|$. $|p_i|\rightarrow \infty$ implies that $\rho_i\rightarrow\infty$. Let $y$ be a limit point of $y_i$. Then  $y\in\lim_{\rho\rightarrow\infty}\rho^{-1}\Sigma\cap U$ and $\lim_{\rho\rightarrow\infty}F_{\rho y,1}(\Sigma)=1$.

Following the blow-up process in Theorem \ref{Theorem: blow up shrinker to get translator}, we define $\widetilde \Sigma_i=\frac{|p_i|}{2}(\Sigma_i- p_i)$. Then $\widetilde \Sigma_i$ locally smoothly converges to a translating soliton, denoted by $\widetilde{\Sigma}$. Moreover, $H(p_i)=1/2|p_i^\bot|=1/2|p_i|$ implies that $|A|(p_i)\geq 1/2|p_i|$, and Theorem \ref{Theorem: blow up shrinker to get translator} implies that $\widetilde \Sigma$ is a Bowl soliton.

On the other hand, by Proposition \ref{prop:entropy upper bound},
\[F_{x,t}(\widetilde \Sigma)\leq \limsup_{\rho\rightarrow\infty}F_{\rho y,1}(\Sigma)\leq 1.\]
Such an upper bound leads to a contradiction to $\widetilde \Sigma$ being a Bowl soliton. Thus we complete the proof.
\end{proof}

\section{Asymptotic to a cylinder}\label{section:cylinder}
We now prove the second part of Theorem \ref{thm:shrinker converge}. In this section, $U$ is always assumed to be a subset of $S^2$ such that $U\cap \lim_{\rho\rightarrow \infty}\rho^{-1}\Sigma$ has only one point $y$. Take $R$ large enough such that $\Sigma\cap\mathcal C_{U,R}$ is diffeomorphic to $S^{1}\times[0,+\infty)$. It follows from Lemma \ref{lem:end curve hausdorff converges} that $\lim_{\rho\rightarrow\infty}\rho^{-1}\Sigma_i\cap U$ consists of disjoint simple closed curves and points for $i$ sufficiently large. 

\bigskip
Given a point $x\neq 0$ in $\mathbb R^3$, we denote by $P(x)$ the hyperplane containing $x$ meanwhile being perpendicular to the vector $x$. We will use the following notation to denote the disk of plane $P(x)$:
\[B_r^x(y):=B_r(y)\cap P(x).\]

The purpose of this section is to prove that $\Sigma_i\cap\mathcal C_{U,R}$ is diffeomorphic to $S^1\times[0,+\infty)$ or a disk. The proof is divided into three cases by the types of $\mathcal C_U\cap\Sigma_i$ (see \ref{def of C_U} for notations): $\lim_{\rho\rightarrow\infty}\rho^{-1}\Sigma_i\cap U$ is either empty or one of the connected components of $\lim_{\rho\rightarrow\infty}\rho^{-1}\Sigma_i\cap U$ is a point, or a curve. %So far, we don't know that $\lim_{\rho\rightarrow\infty}\rho^{-1}\Sigma_i\cap U$ contains only one connected component)

\subsection{$\Sigma_i\cap\mathcal C_U$ contains a cylindrical end}
Firstly, we consider that $\lim_{\rho\rightarrow\infty}\rho^{-1}\Sigma_i\cap U$ always contains a connected component which is a single point $y_i$. Then by Lemma \ref{lem:end curve hausdorff converges}, $y_i\rightarrow y$ as $i\rightarrow\infty$. The goal of this section is to prove the following:

\begin{proposition}\label{prop:cylinder end}
For sufficiently large $i$, $\Sigma_i\cap\mathcal C_{U,R}$ is diffeomorphic to $\Sigma\cap\mathcal C_{U,R}$. 
\end{proposition} 
\begin{proof}
Let us first sketch the idea. We want to show that $\Sigma_i\cap\mathcal C_{U,R}$ is topologically a cylinder by contradiction. Suppose $\Sigma_i\cap\mathcal C_{U,R}$ is not topologically a cylinder near the point $p_i$, and the projection of $p_i$ to $y_i$ is $\rho_iy_i$. Then we can translate $\Sigma_i$ by the vector $-\rho_iy_i$, in other words we consider the sequence $\Sigma_i-\rho_iy_i$. We will show that $\Sigma_i-\rho_iy_i$ converges to a cylinder smoothly. As a result, $\Sigma_i$ should be topologically a cylinder near $p_i$, which is a contradiction.

The proof is divided into four steps.

\vspace{1em}
{\noindent\bf Step 1:}
{\em For any sequence $\rho_i\rightarrow\infty$, $F_{\rho_iy_i,1}(\Sigma_i)\rightarrow \lambda_1$.}

\begin{proof}[Proof of Step 1]
On one hand,
\[\liminf_{i\rightarrow\infty}F_{\rho_iy_i,1}(\Sigma_i)
\geq\liminf_{i\rightarrow\infty}\liminf_{\rho\rightarrow\infty}F_{\rho y_i,1}(\Sigma_i)
=\lambda_1.\]
On the other hand, by Proposition \ref{prop:entropy upper bound} (taking $a_i\equiv 1$,  $x=0$, $t=1$),
\[\limsup_{i\rightarrow\infty} F_{\rho_iy_i,1}(\Sigma_i)
\leq \limsup_{\rho\rightarrow\infty}F_{\rho y,1}(\Sigma)
= \lambda_1.\]
Then the desired result follows.
\end{proof}

{\noindent\bf Step 2:}
{\em For any sequence $\rho_i\rightarrow\infty$, $\Sigma_i-\rho_iy_i$ locally smoothly converges to $S^1(\sqrt 2)\times\mathbb R_y$ up to a subsequence.}
 
\begin{proof}[Proof of Step 2]
Define the Brakke flow for each $i$ and $t<0$,
\[ \mu_t^i=\mathcal H^2\lfloor\big(\sqrt{-t}\Sigma_i-\rho_iy_i\big).\]

By the compactness theorem of Brakke flow, there exists a Brakke flow $\{\nu_t\}_{t<0}$ such that for each $t<0$, $\mu_t^i\rightarrow \nu_t$ as varifolds. In the following, we show that the support of $\nu_{-1}$ is a self-shrinker.

By the monotonicity of $F$-functional, $\lambda(\nu_t)\leq \lambda_1$. Meanwhile, for any $t<0$,
\begin{align*}
\int \Phi(x,t)d\nu_t &=\lim_{i\rightarrow\infty}\int \Phi(x,t)d\mu_t^i\\
                     &=\lim_{i\rightarrow\infty}\int_{\sqrt{-t}\Sigma_i-\rho_iy_i}\Phi(x,t)dx\\
                     &=\lim_{i\rightarrow\infty}\int_{\Sigma_i-\frac{\rho_i}{\sqrt{-t}}y_i}\Phi(x,-1)dx\\
                     &=\lim_{i\rightarrow\infty}F_{\frac{\rho_i}{\sqrt{-t}}y_i,1}(\Sigma_i)\\
                     &=\lambda_1.
\end{align*}
By the monotonicity formula for Brakke flows, $\nu_{-1}$ is $F$-stationary, i.e.
\[\overrightarrow{H}+\frac{S(x)^\perp\cdot x}{2}=0,\ \ \ \nu_{-1}-a.e. \ x.\]
Here $S(x)^{\perp}$ is the projection operator, sending $x$ to the normal direction of $x$. In other words, $\nu_{-1}$ is a self-shrinker in the sense of varifolds. Now we show that $\nu_{-1}$ is supporting on a smooth self-shrinker.

For $\tau\in \mathbb R$, $\rho_i+\tau\rightarrow+\infty $ as $\rho_i\rightarrow+\infty$. Then $\nu_{-1}-\tau y$, which is the limit of $\mu_{-1}^i-\tau y=\mathcal H^2\lfloor\big(\Sigma_i-(\rho_i+\tau)y_i\big)$, is also $F$-stationary. Thus $S(x)^{\perp}\cdot y=0$, which implies that $\nu_{-1}$ splits off a line (c.f. \cite{Zhu16} Lemma 1.10). Recall that $F_{0,1}(\nu_{-1})=\lambda_1$. We conclude that $\nu_{-1}=\mathcal H^2\lfloor S^{1}(\sqrt 2)\times \mathbb R_y$.

By the local regularity of Brakke flows \cite{Whi02} (see Theorem \ref{thm:Brakke regularity}), $\Sigma_i-\rho_iy_i$ locally smoothly converges to $S^1(\sqrt 2)\times\mathbb R_y$.
\end{proof}

{\noindent\bf Step 3:}
{\em For sufficiently large $i$, $\Sigma_i\cap\mathcal C_{U,R}\subseteq N_i$, where $N_i=\bigcup_{\rho=R}^{+\infty}B_2^{\rho y_i}(\rho y_i)$.}

Here $N_i$ is a part of the solid cylinder with direction $y_i$ and radius $2$.
\begin{proof}[Proof of Step 3]
We prove by contradiction. Assume on the contrary that there is a subsequence of $\{\Sigma_i\}$ (still denoted by $\{\Sigma_i\}$) such that $\Sigma_i\cap\mathcal C_{U,R}$ does not completely lie in $N_i$. Then for each $i$, $\Sigma_i\cap\bigcup_{\rho=R}^{+\infty}\partial B_2^{\rho y_i}(\rho y_i)\neq\emptyset$. Take $p_i$ in this set and let $\rho_i=\langle p_i,y_i\rangle$ (see Figure \ref{fig:asymptotic to a cylinder}). 

\begin{figure}[h]
\begin{center}
\def\svgwidth{0.8\columnwidth}
  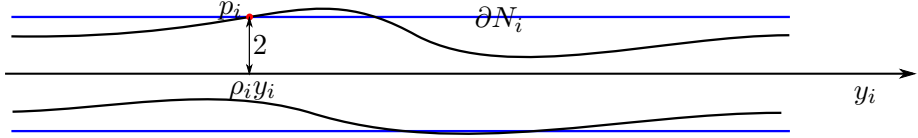
  \caption{Asymptotic to a cylinder.}
  \label{fig:asymptotic to a cylinder}
\end{center}
\end{figure}

If $\rho_i$ is bounded, take a subsequence of $p_i$ (still denoted by $p_i$) such that $\lim\rho_i=\rho_\infty$. Then $\Sigma_i-\rho_iy_i$ locally smoothly converges to $\Sigma-\rho_\infty y$, which is also close to $S^1(\sqrt 2)\times \mb R_y$. Hence for $i$ sufficiently large,  $\Sigma_i\cap \bigcup_{\rho=\rho_i-1}^{\rho_i+1}\partial B_2^{\rho y_i}(\rho y_i)=\emptyset$. This is a contradiction.

If $\rho_i$ is unbounded, take a subsequence of $p_i$ (still denoted by $p_i$) such that $\lim\rho_i=+\infty$. Step 2 gives that $\Sigma_i-\rho_iy_i$ locally smoothly converges to $S^1(\sqrt{2})\times \R_y$ up to a subsequence (still denoted by $\{\Sigma_i\}$). However, $p_i\in \partial B_2^{\rho_iy_i}(\rho_iy_i)$ implies that $\Sigma_i-\rho_iy_i$ can not converge to $S^1(\sqrt{2})\times \R_y$ locally smoothly. This is a contradiction.

Therefore, $\Sigma_i\cap\mathcal C_{U,R}\subseteq N_i$ for large $i$.
\end{proof}

Now we can prove our main theorem in this subsection by contradiction. Suppose on the contrary that there is a subsequence of $\Sigma_i\cap\mathcal C_{U,R}$ which are not diffeomorphic to $\Sigma\cap\mathcal C_{U,R}$ (still denoted by $\Sigma_i$). Then we can take $r_i\rightarrow\infty$ so that $\Sigma_i$ is not diffeomorphic to $S^1\times [0,1]$ near $r_iy_i$. Applying the result in Step 2 gives a contradiction. We present the details below.

\vspace{1em}
{\noindent\bf Step 4:}
{\em $\Sigma_i\cap\mathcal C_{U,R}$ is diffeomorphic to $\Sigma\cap \mathcal C_{U,R}$.}
\begin{proof}[Proof of Step 4]
By Step 3, it suffices to show that there exists $I>0$ such that for any $i>I$ and $r>R$, $\Sigma_i\cap\bigcup_{\rho=r}^{r+1}B_2^{\rho y_i}(\rho y_i)$ is diffeomorphic to $S^1\times [0,1]$.

Assume the statement is not true. Then for each $i$, take $r_i>R$ such that $\Sigma_i\cap\bigcup_{\rho=r_i}^{r_i+1}B_2^{\rho y_i}(\rho y_i)$ is not diffeomorphic to $S^1\times [0,1]$. 

If $r_i$ is bounded, take a subsequence of $r_i$ (still denoted by $r_i$) such that $\lim r_i=r$. Then $\Sigma_i-r_iy_i$ locally smoothly converges to $\Sigma-r y$, which is a contradiction.

If $r_i$ is unbounded, take a subsequence of $r_i$ (still denoted by $r_i$) such that $\lim r_i=+\infty$. Because the topological type is different, $\Sigma_i-r_iy_i$ can not locally smoothly converges to  $S^1(\sqrt 2)\times \mathbb R_y$, which contradicts the conclusion in Step 2.
\end{proof}
Combining all the steps above shows the proposition.
\end{proof}

%\bigskip
\vspace{0.8em}
\subsection{$\Sigma_i\cap\mathcal C_U$ contains a conical end}
Now we consider that $\lim_{\rho\rightarrow\infty}\rho^{-1}\Sigma_i\cap U$ contains a connected component which is a simple closed curve. 

Let $\gamma_i$ be a connected component of $\lim_{\rho\rightarrow\infty}\rho^{-1}\Sigma_i\cap U$. By Lemma \ref{lem:end curve hausdorff converges}, $\gamma_i\rightarrow y$ in Hausdorff distance. Denote by $\widetilde \gamma_i$ the curve $\mathcal C_{\gamma_i}\cap B_1^{y}(y)$. In the following of this section, $\wti\gamma_i$ is also regarded as a curve in $\mb R^2 (=P(y))$.

Before we state our main proposition, we need some Lemmas. We want to show that $\{\widetilde\gamma_i\}$ is a entropy minimizing sequence among all close plane curves.

\begin{lemma}\label{Lem: entropy of cone gamma equal to entropy of gamma taking limit}
\[\lim_{i\to\infty}\lambda(\widetilde\gamma_i)=\lim_{i\to\infty}\lambda(\mathcal{C}_{\widetilde\gamma_i})=\lambda_1.\]
\end{lemma}

\begin{proof}
On one hand, take $s_i>0$ and $y_i\in S^2(1)$ such that
\[\lambda(\mathcal{C}_{\widetilde\gamma_i})=F_{s_iy_i,1}(\mathcal{C}_{\widetilde\gamma_i}).\]
$\gamma_i$ converges to a point implies that $s_i\rightarrow \infty$. Then Proposition \ref{prop:entropy upper bound} and the monotonicity formula imply that for an increasing to infinity sequence $\rho_i$,
\[
\begin{split}
\lambda_1 &=\lim_{\rho\rightarrow\infty}F_{\rho y,1}(\Sigma)
\geq 
\lim_{i\to\infty}F_{\rho_i y_i,1+s_i^{-2}\rho_i^2}(\Sigma_i)\\
&\geq 
\lim_{i\to\infty}\lim_{\rho\to\infty}F_{\rho y_i,1+s_i^{-2}\rho^2}(\Sigma_i)\\
&= \lim_{i\to\infty} F_{0,1}(\mathcal{C}_{\widetilde\gamma_i}-s_iy_i)=\lim_{i\to\infty}\lambda(\mathcal{C}_{\widetilde\gamma_i}).
\end{split}
\]
Therefore, 
\begin{equation}\label{eq:cone less than cylinder}
\lim_{i\to\infty}\lambda(\mathcal{C}_{\widetilde\gamma_i})\leq \lambda_1\leq \lim_{i\to\infty}\lambda(\widetilde\gamma_i).
\end{equation}

On the other hand, take $x_i\in B_1^{y_i}(y_i)$ and $a_i\in(0,\infty)$ such that 
\[\lambda(\widetilde\gamma_i)=F_{0,1}(a_i^{-1}(\widetilde\gamma_i-x_i)).\]
Then $a_i$ must converge to $0$ since the diameter of $\widetilde\gamma_i$ converges to $0$. Then in any compact domain, the varifold distance between $a_i^{-1}(\mathcal{C}_{\widetilde\gamma_i}-x_i)$ and $a_i^{-1}(\widetilde\gamma_i-x_i)\times\mb R$ goes to zero. Since the entropy of $\Sigma_i$ are uniformly bounded, so does the entropy of $\mathcal{C}_{\widetilde\gamma_i}$ and $\widetilde\gamma_i$. Therefore we conclude that
\[
|F_{0,1}(a_i^{-1}(\mathcal{C}_{\widetilde\gamma_i}-x_i))-\lambda(\widetilde\gamma_i))|=|F_{0,1}(a_i^{-1}(\mathcal{C}_{\widetilde\gamma_i}-x_i))-F_{0,1}(a_i^{-1}(\widetilde\gamma_i-x_i)\times\R)|\rightarrow 0.
\]
This deduces
\[\lim_{i\to\infty}\lambda(\mathcal{C}_{\widetilde\gamma_i})\geq \lim_{i\to\infty}\lambda(\widetilde\gamma_i).\]
Together with \eqref{eq:cone less than cylinder}, we get the desired equality.
\end{proof}

With the entropy minimizing property, we prove that after normalizing, these curves converges to a round circle in the sense of varifolds. The argument here is the one-dimensional case of \cite{BW18}.
\begin{lemma}\label{Lem:2sqrt2gamma i converges to S1}
There exist $x_i\in \mb R^2$ and $T_i>0$ so that $T_i^{-1/2}(\wti\gamma_i-x_i)$ converges to $S^{1}(\sqrt 2)$ in the sense of varifolds.
\end{lemma}

\begin{proof}
Let $\{\gamma_i(t)\}_{t\in[0,T_i)}$ be the maximal smooth curve shortening flow with $\gamma_i(0)=\gamma_i$. By the entropy bound and \citelist{\cite{GH86}\cite{Grayson}}, the first singularity of this flow is at $(x_i,T_i)$, where the flow disappears in a round point. Now let $\Gamma_i(t)=T_i^{-1/2}(\wti\gamma_i((t+1)T_i)-x_i)$ for $t\in[-1,0)$. We can follow the argument in \cite{BW18} by replacing \cite{BW18}*{Proposition 3.2} with Proposition \ref{prop:smooth for all} to prove that $\Gamma_i(t)$ actually have uniformly curvature bound for each fixed $t\in(-1,0)$. 

By Brakke's compactness theorem, up to a subsequence $\{\Gamma_i(t)\}_{t\in[-1,0)}$ converges to a Brakke flow $\{\nu_t\}_{t\in[-1,0)}$ in the sense of varifolds. By the uniform curvature bounds for each fixed $t\in(-1,0)$, $\nu_t$ is a smooth curve. Clearly, $\lambda(\nu_t)=\lambda_1$ and, by the upper semi-continuous of Gaussian density, this flow becomes singular at $(0,0)$. Hence $\nu_{-1/2}$ must be a smooth circle with radius $1$, which implies that $\nu_{-1}$ is a round circle with radius $\sqrt 2$.
\end{proof}

\begin{proposition}\label{prop:smooth for all}
Suppose $\{\Gamma(t)\}_{t\in\R}$ is an eternal curve shortening flow in the plane $\R^2$ with $\sup_{t\in\mb R}\lambda(\Gamma(t))<2$. Then $\Gamma(t)$ is a straight line for any $t\in\R$.
\end{proposition}
\begin{proof}
We follow the proof of \cite{BW18}*{Proposition 3.2}. Again we consider the blow-down sequence of curve shortening flow $\{\rho^{-1}\Gamma(\rho^{2}t)\}_{t\in\mb R}$. Then a subsequence converges to a self-shrinking Brakke flow $\{\nu_t\}_{t\in\mb R}$, with entropy less than $2$. Moreover, the entorpy bound and \cite{Whi09} suggest that the $\nu_{-1}$ must be a $1$-dimensional $F$-stationary varifold without triple junction. Thus $\nu_{-1}$ can only be either a circle or a straight line. If it is a circle, then the smooth convergence given by Brakke's regularity theorem, $\rho^{-1}\Gamma(-\rho^2)$ is contained in $S^1(2)$ for large $\rho$. However, this contradicts the hypothesis that $\{\Gamma(t)\}$ is a smooth flow for all $t\in\mb R$ and then we must have that $\nu_{-1}$ is a straight line. 

As such, by the monotonicity formula, $\lambda(\Gamma(t))=1$. Then we conclude the proposition.
\end{proof}

Let $x_i$ and $T_i$ be the point and real number in Lemma \ref{Lem:2sqrt2gamma i converges to S1}. Recall that $x_i$ is a point in $\mb R^2=P(y)\subset \mb R^3$. Let $y_i=x_i/|x_i|$ and $\beta_i=T_i^{1/2}/|x_i|$. Note that $T_i^{-1/2}(\wti\gamma_i-x_i)\times \mb R$ converges to $S^1(\sqrt 2)\times \mb R$ in the sense of varifolds. Recall that in any fixed compact domain, the varifold distance between $T_i^{-1/2}(\mathcal{C}_{\widetilde\gamma_i}-x_i)$ and $T_i^{-1/2}(\widetilde\gamma_i-x_i)\times\mb R$ goes to zero. Then we conclude that $\mc C_{\gamma_i}-\beta_i^{-1}y_i$ converges to $S^1(\sqrt 2)\times \mb R$ in the sense of varifolds.

\begin{proposition}\label{prop:cone}
$\Sigma_i\cap\mathcal C_{U,R}$ is diffeomorphic to $\Sigma\cap\mathcal C_{U,R}$.
\end{proposition}
\begin{proof}
We use the similar argument as in the proof of Proposition \ref{prop:cylinder end}. However, there is a difference. For the cylindrical ends case, we only need translations to bring the point where the topology of the end is bad, to a neighborhood of the origin to argue by a contradiction; for the conical ends case, we need not only translations but also dilations to bring the point, where the topology of the end is bad, back to a neighborhood of the origin and then argue by contradiction. The proof is also divided into four steps:

\vspace{0.5em}
{\noindent\bf Step 1:}
{\em For any sequence $\rho_i\rightarrow\infty$, we have
\[F_{\rho_iy_i,1+\beta_i^2\rho_i^2}(\Sigma_i)\rightarrow \lambda_1.\]}
\begin{proof}[Proof of Step 1]
By the monotonicity of $F$-functional,
\begin{align*}
\liminf_{i\rightarrow\infty}F_{\rho_iy_i,1+\beta_i^2\rho_i^2}(\Sigma_i)
&\geq\liminf_{i\rightarrow\infty}\liminf_{\rho\rightarrow\infty}F_{\rho y_i,1+\beta_i^2\rho^2}(\Sigma_i)\\
&=\liminf_{i\rightarrow\infty}\liminf_{\rho\rightarrow\infty}F_{0,1}((\rho \beta_i)^{-1}\Sigma_i-\beta_i^{-1}y_i)\\
&=\liminf_{i\rightarrow\infty} F_{0,1}(\mathcal C_{\gamma_i}-\beta_i^{-1}y_i)=\lambda_1.
\end{align*} 
By Proposition \ref{prop:entropy upper bound},
\[
\limsup_{i\rightarrow\infty}F_{\rho_iy_i,1+\beta_i^2\rho_i^2}(\Sigma_i)
\leq \liminf_{\rho\rightarrow\infty} F_{\rho y,1}(\Sigma)=\lambda_1.
\]
Then the desired result follows immediately.

\end{proof}

{\noindent\bf Step 2:}
{\em For any $\rho_i\rightarrow\infty$ with $\lim_{i\rightarrow\infty}\rho_i\beta_i<+\infty$,  $\frac{\Sigma_i-\rho_iy_i}{\sqrt{1+\beta_i^2\rho_i^2}}$ locally smoothly converges to $S^1(\sqrt 2)\times \mathbb R_y$ up to a subsequence.}

\begin{remark}
Intuitively, such a sequence is ``bounded" by $\mc C_{\gamma_i}-\beta_i^{-1}y_i$ and $\Sigma-\rho y$. However, both sequences locally smoothly converge to the round cylinder $S^1(\sqrt 2)\times \mb R_y$.
\end{remark}

\begin{proof}[Proof of Step 2]
For any sequence of $\rho_i\rightarrow \infty$ with $\lim_{i\rightarrow\infty}\rho_i\beta_i<+\infty$, there exists a constant $K>0$ so that $\rho_i\beta_i<K$ for all large $i$. Define the Brakke flow for $t\leq  -K^2/(1+K^2)$ by
\[ \mu_t^i=\mathcal H^2\lfloor\Big(\left(\sqrt{\frac{1}{1+\beta_i^2\rho_i^2}-1-t}\,\right)\Sigma_i-\frac{\rho_i}{\sqrt{1+\beta_i^2\rho_i^2}}y_i\Big).\]
This is the mean curvature flow so that $\mu_{-1}^i=\frac{\Sigma_i-\rho_iy_i}{\sqrt{1+\beta_i^2\rho_i^2}}$. Then by the compactness of Brakke flows (see Lemma \ref{lem:brakke compactness}), $\mu_t^i$ converges to a Brakke flow $\nu_t$ in the sense of Radon measure. We want to show that the support of $\nu_{-1}$ is a self-shrinker.

By the monotonicity of $F$-functional, $\lambda(\nu_t)\leq \lambda_1$. By computation,
\begin{equation}
\mu_t^i=\mathcal H^2\lfloor\Big(\sqrt{-t}\big(\frac{\Sigma_i-\widetilde \rho_iy_i}{\sqrt{1+\beta_i^2\widetilde\rho_i^2}}\big)\Big),
\end{equation}
where $\widetilde \rho_i=\rho_i\Big/\sqrt{1-(1+t)(1+\beta_i^2\rho_i^2)}$. Note that for any $t< -K^2/(1+K^2)$, we have $\widetilde \rho_i\rightarrow \infty$ since $\rho_i\rightarrow\infty$ and $\beta_i\rightarrow 0$. Therefore,
\begin{align*}
\int \Phi(x,t)d\nu_t&=\lim_{i\rightarrow\infty}\int \Phi(x,t)d\mu_t^i\\
&=\lim_{i\rightarrow\infty}F_{\widetilde \rho_iy_i,1+\beta_i^2\widetilde \rho_i^2}(\Sigma_i)\\
&=\lambda_1.
\end{align*}
By the monotonicity formula, $\nu_{-1}$ is $F$-stationary, i.e. 
\[\overrightarrow{H}+\frac{S(x)^\perp\cdot x}{2}=0,\ \ \ \nu_{-1}-a.e.\ x.\]

We now show that $\nu_{-1}$ is supporting on a smooth self-shrinker. For any $\tau\in\mathbb R$, set 
\[\rho'_i=\rho_i+\tau\sqrt{1+\beta_i^2\rho_i^2} \text{\ \ \ and\  \ \ } s_i=\frac{1+\beta_i^2\rho_i'^2}{1+\beta_i^2\rho_i^2}.\]
It follows that $\rho_i'\rightarrow\infty$, $\lim_{i\rightarrow\infty}\beta_i\rho_i'<+\infty$ and $s_i\rightarrow 1$ as $i\rightarrow\infty$. Then $\nu_{-1}-\tau y$, the limit of $\mu_{-1}^i-\tau y_i=\mathcal H^2\lfloor\Big(\sqrt{s_i}\big(\frac{\Sigma_i- \rho'_iy_i}{\sqrt{1+\beta_i^2 \rho_i'^2}}\big)\Big)$, is also $F$-stationary. This shows that $S(x)^\perp\cdot y=0$ for almost $x$. Therefore, $\nu_{-1}$ splits off a line with direction $y$. Note that $\lambda(\nu_{-1})=\lambda_1$. Hence $\nu_{-1}$ supports on $S^1(\sqrt2)\times \mathbb R_y$ with multiplicity $1$. Then by local regularity for Brakke flows (Theorem \ref{thm:Brakke regularity}), we conclude that  $\frac{\Sigma_i-\rho_iy_i}{\sqrt{1+\beta_i^2\rho_i^2}}$ locally smoothly converges to $S^1(\sqrt 2)\times \mathbb R_y$ up to a subsequence.
\end{proof}

{\noindent\bf Step 3:}
{\em For sufficiently large $i$, $\Sigma_i\cap \mathcal{C}_{U,R}\subseteq N_i$, where $N_i=\bigcup_{\rho=0}^{+\infty}B_{2\sqrt{1+\beta_i^2\rho^2}}^{\rho y_i}(\rho y_i)$.}
\begin{proof}[Proof of Step 3]
Assume on the contrary that a subsequence of $\Sigma_i$ is not contained in $N_i$. We still denote the subsequence by $\{\Sigma_i\}$. Then for each $i$,  $\Sigma_i\cap \mathcal{C}_{U,R}\cap\partial N_i$ is non-empty. Take $p_i$ in this intersection and set $\rho_i=\langle p_i,y_i\rangle$. 

\medskip
If $\rho_i$ is bounded, take a subsequence of $p_i$ (still denoted by $p_i$) such that $\lim\rho_i=\rho$. Then 
\[\lim_{i\rightarrow\infty}\Bigg|\frac{p_i-\rho_iy_i}{\sqrt{1+\beta_i^2\rho_i^2}}\Bigg|=\lim_{i\rightarrow\infty}\big|p_i-\rho_iy_i\big|=2.\]
Recall that $\Sigma_i-\rho_iy_i$ locally smoothly converges to a smooth surface. Therefore the limit surface must intersect $\mb R_y\times S^1(2)$. This contradicts the fact that $\Sigma_i-\rho_iy_i$ locally smoothly converges to $\Sigma-\rho y$.

\medskip
If $\rho_i$ is unbounded and $\rho_i\beta_i$ is bounded, take a subsequence of $p_i$ (still denoted by $p_i$) such that $\lim_{i\rightarrow\infty}\rho_i=+\infty$ and $\lim_{i\rightarrow\infty}\beta_i\rho_i<+\infty$. Then $p_i\in\Sigma_i\cap\partial N_i$ also implies that
\[\lim_{i\rightarrow\infty}\Bigg|\frac{p_i-\rho_iy_i}{\sqrt{1+\beta_i^2\rho_i^2}}\Bigg|=2,\]
which contradicts that $\frac{\Sigma_i-\rho_iy_i}{\sqrt{1+\beta_i^2\rho_i^2}}$ locally smoothly converges to $S^1(\sqrt 2)\times \mathbb R_y$ by Step 2.

\medskip
It remains to consider $\rho_i\beta_i\rightarrow\infty$ (see Figure \ref{fig:small cones}). Denote by $z_i=p_i/\rho_i$. Then for any $k>0$,
\begin{align*}
\lim_{i\rightarrow\infty}\frac{\dist(k\beta_i^{-1}z_i,\mb R_{y_i})}{\sqrt{1+k^2}}&=\lim_{i\rightarrow\infty}\frac{\dist(\rho_iz_i,\mb R_{y_i})}{\sqrt{1+k^2}}\cdot \frac{k\beta_i^{-1}}{\rho_i}\\
&=\lim_{i\rightarrow\infty}\frac{\dist(p_i,\mb R_{y_i})}{\sqrt{1+\rho_i^2\beta_i^2}}\cdot\frac{\sqrt{1+\rho_i^2\beta_i^2}}{\sqrt{1+k^2}}\cdot \frac{k}{\rho_i\beta_i}\\
&=\frac{2k}{\sqrt{1+k^2}}.
\end{align*}

\begin{figure}[h]
\begin{center}
\def\svgwidth{0.8\columnwidth}
  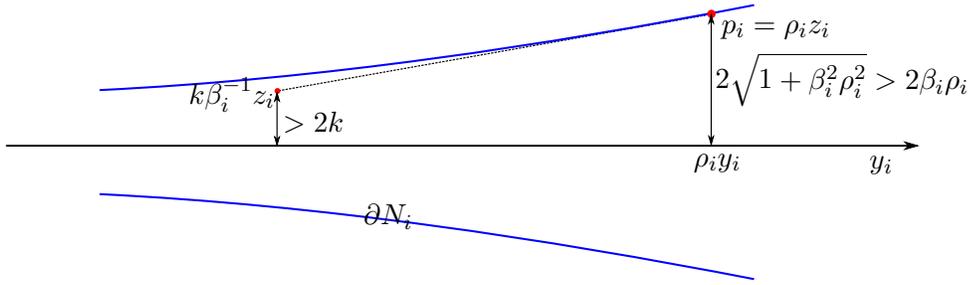
  \caption{Small cones.}
  \label{fig:small cones}
\end{center}
\end{figure}

Recall that $\frac{\Sigma_i-k\beta_i^{-1}z_i}{\sqrt{1+k^2}}$ locally smoothly converges to $S^{1}(\sqrt 2)\times \mb R_y$. Thus we must have that for large $i$,
\begin{equation}\label{eq:far away from Sigmai}
B_{k/4}(k\beta_i^{-1}z_i)\cap \Sigma_i=\emptyset.
\end{equation}
Take $K$ large enough and fixed, then the uniform entropy bound and \eqref{eq:far away from Sigmai} give that 
\begin{equation}\label{eq:entropy small}
F_{K\beta_i^{-1}z_i,1}(\Sigma_i)<1/2.
\end{equation}

Now we derive a entropy lower bound to get a contradiction. Because $\rho_iz_i\in\Sigma_i$, we always have
\[\lim_{t\rightarrow 0}F_{\rho_iz_i,t}(\Sigma_i)=1.\]
Then by the monotonicity formula for $F$-functional, 
\begin{align*}
\lim_{t\rightarrow 0}F_{\rho_iz_i,t}(\Sigma_i)&=\lim_{t\rightarrow 0}F_{\rho_iz_i,1-\frac{1-t}{\rho_i^2}\cdot\rho_i^2}(\Sigma_i)\\
&\leq \limsup_{i\rightarrow\infty}\lim_{t\rightarrow 0}F_{K\beta_i^{-1}z_i,1-\frac{1-t}{\rho_i^2}\cdot(K\beta_i^{-1})^2}(\Sigma_i)\\
&\leq \limsup_{i\rightarrow\infty}F_{K\beta_i^{-1}z_i,1}(\Sigma_i),
\end{align*}
which contradicts \eqref{eq:entropy small}.

Above all, $\Sigma_i\cap \mathcal{C}_{U,R}\subseteq N_i$ for large $i$.
\end{proof}

We now prove Proposition \ref{prop:cone} by contradiction. Suppose on the contrary that $\Sigma_i\cap \mathcal{C}_{U,R}$ is not diffeomorphic to $\Sigma\cap \mathcal{C}_{U,R}$ for a subsequence of $\Sigma_i$ (still denoted by $\Sigma_i$). Then take $r_i$ so that $\Sigma_i$ is not diffeomorphic to $\Sigma$ in a suitable neighborhood of $r_iy_i$. If $r_i$ is uniformly bounded, then it contradicts the fact that $\Sigma_i$ locally smoothly converges to $\Sigma$. If $r_i$ is unbounded, then by taking the limit, we see a contradiction to the conclusion in Step 2.

{\noindent\bf Step 4:}
{\em $\Sigma_i\cap\mathcal C_{U,R}$ is diffeomorphic to $\Sigma\cap \mathcal C_{U,R}$.}
\begin{proof}[Proof of Step 4]
It suffices to show that there exists $I>0$ such that for any $i>I$ and $r>R$, $\Sigma_i\cap\bigcup_{\rho=r}^{r+1}B_{2\sqrt{1+\beta_i^2\rho^2}}^{\rho y_i}(\rho y_i)$ is diffeomorphic to $S^1\times [0,1]$. Define
\[\Sigma_i(r_i;y_i)=\Sigma_i\cap\bigcup_{\rho=r}^{r+1}B_{2\sqrt{1+\beta_i^2\rho^2}}^{\rho y_i}(\rho y_i).\]

Assume the statement is not true. Then for each $i$, take $r_i>R$ such that $\Sigma_i(r_i;y_i)$ is not diffeomorphic to $S^1\times [0,1]$. 

If $r_i$ is bounded, take a subsequence of $r_i$ (still denoted by $r_i$) such that $\lim r_i=r$. Note that $\Sigma_i(r_i;y_i)-r_iy_i$ smoothly converges to $(\Sigma-ry)\cap B_2(0)\times [0,1]$, which leads to a contradiction.

If $r_i$ is unbounded and $r_i\beta_i$ is bounded, take a subsequence of $r_i$ (still denoted by $r_i$) such that $\lim r_i=+\infty$ and $\lim r_i\beta_i<+\infty$. Then we have proved in Step 3 that $\frac{\Sigma_i-r_iy_i}{\sqrt{1+\beta_i^2 r_i^2}}$ smoothly converges to $S^1(\sqrt 2)\times \mathbb R$. It follows that $\frac{\Sigma_i(r_i;y_i)-r_iy_i}{\sqrt{1+\beta_i^2 r_i^2}}$ smoothly converges to $S^1(\sqrt 2)\times \mathbb [0,1]$, which leads to a contradiction.

It remains to consider $r_i\beta_i\rightarrow\infty$. Then the function $|x|^2$ must have critical points on $\Sigma_i(r_i;y_i)$. Without loss of generality, denote by $\rho_iz_i\in \Sigma_i(r_i;y_i)$ ($|z_i|=1$) the critical point of $|x|^2$. It follows that $\rho_i\beta_i\rightarrow\infty$. By Theorem \ref{Theorem: blow up shrinker to get translator} and \ref{Theorem: Asymptotic to a cone main theorem}, $\frac{\rho_i}{2}(\Sigma_i-\rho_iz_i)$ locally smoothly converges to a Bowl soliton. Thus there exists $(x,t)\in\mb R^3\times (0,\infty)$ (see \cite{Gua16}) so that 
\begin{equation}\label{eq:F lower bound}
\liminf_{i\rightarrow\infty}F_{x,t}(\rho_i(\Sigma_i-\rho_iz_i))>\frac{1+\lambda_1}{2}.
\end{equation}
However, the monotonicity formula of $F$-functional gives that 
\begin{align*}
\lim_{i\rightarrow\infty}F_{x,t}(\rho_i(\Sigma_i-\rho_iz_i)) &=\lim_{i\rightarrow\infty}F_{\rho_i(z_i+x/\rho^2_i),1+(t/\rho_i^4-1/\rho_i^2)\cdot\rho_i^2}(\Sigma_i)\\
&\leq\limsup_{k\rightarrow\infty}\limsup_{i\rightarrow\infty} F_{k\beta_i^{-1}(z_i+x/\rho^2_i),1+(t/\rho_i^4-1/\rho_i^2)\cdot(k\beta_i^{-1})^2}(\Sigma_i)\\
&=\limsup_{k\rightarrow\infty}\limsup_{i\rightarrow\infty}F_{k\beta_i^{-1}z_i,1}(\Sigma_i)\\
&=\limsup_{k\rightarrow\infty}\limsup_{i\rightarrow\infty}F_{k\beta_i^{-1}(z_i-y_i),1}(\Sigma_i-k\beta_i^{-1}y_i).
\end{align*}
A direct computation gives that, for large $i$,
\begin{align*}
k\beta_i^{-1}\cdot|z_i-y_i|&=k\beta_i^{-1}\cdot|\rho_iz_i-\rho_iy_i|/\rho_i\leq k\beta_i^{-1}\cdot 2\dist(p_i,\mb R_{y_i})/\rho_i\\
&\leq k\beta_i^{-1}\cdot 4\sqrt{1+\beta_i^2\rho_i^2}/\rho_i\leq 8k,
\end{align*}
where we used the fact that $z_i\in \Sigma_i(r_i;y_i)$. Hence we have that 
\begin{align*}
\lim_{i\rightarrow\infty}F_{x,t}(\rho_i(\Sigma_i-\rho_iz_i))&\leq \limsup_{k\rightarrow\infty}\limsup_{i\rightarrow\infty}F_{k\beta_i^{-1}(z_i-y_i),1}(\Sigma_i-k\beta_i^{-1}y_i)\\
&\leq \limsup_{k\rightarrow\infty}\limsup_{i\rightarrow\infty}\sup_{z\in B_{8k}(0)}F_{z,1}(\Sigma_i-k\beta_i^{-1}y_i)\\
&\leq\limsup_{k\rightarrow\infty}\sup_{z\in B_{8k}(0)}F_{z,1}(S^{1}(\sqrt{2(1+k^2)}\times \mb R)=1,
\end{align*}
where we used that $\frac{\Sigma_i-k\beta_i^{-1}y_i}{\sqrt{1+k^2}}$ locally smoothly converges to $S^1(\sqrt 2)\times \mb R_y$. This entropy upper bound contradicts \eqref{eq:F lower bound}. Then the proof of Step 4 is finished.
\end{proof}
We have completed the proof of the proposition.
\end{proof}

\bigskip
\subsection{$\Sigma_i\cap\mathcal C_U$ is compact}
In this part, we consider that $\lim_{\rho\rightarrow\infty}\rho^{-1}\Sigma_i\cap U=\emptyset$. Then $\Sigma_i\cap \mathcal C_U$ is compact for all $i$. Take $q_i\in\Sigma_i\cap \mathcal C_U$ such that 
\[ |q_i|=\sup_{x\in\Sigma_i\cap \mathcal C_U}|x|.\]

Set $y_i=q_i/|q_i|$ and $d_i=|q_i|$. Then $d_i\rightarrow+\infty$ and $y_i\rightarrow y$.

\begin{proposition}\label{prop:compact}
$\Sigma_i\cap\mathcal C_{U,R}$ is diffeomorphic to a disk.
\end{proposition}
\begin{proof}
The proof is divided into four steps:

{\noindent\bf Step 1:}
{\em For any $\rho_i\rightarrow+\infty$ with $d_i(d_i-\rho_i)\rightarrow+\infty$, $F_{\rho_iy_i,1-\rho_i^2/d_i^2}(\Sigma_i)\rightarrow\lambda_1$.}

\begin{proof}[Proof of Step 1]
By Theorem \ref{Theorem: blow up shrinker to get translator}, $\frac{d_i}{2}(\Sigma_i-d_iy_i)$ locally smoothly converges to a translating soliton $\Sigma_y$ with translating direction $-y$. According to the entropy bound, it is a Bowl soliton by Theorem \ref{thm:translator low entropy}. Hence $\frac{\Sigma_y+\rho y}{\sqrt \rho}$ locally smoothly converges to $S^1(\sqrt2)\times \mathbb R_y$ as $\rho\rightarrow+\infty$. By the monotonicity of $F$-functional,
\begin{align*}
\liminf_{i\rightarrow\infty}F_{\rho_iy_i,1-\frac{\rho_i^2}{d_i^2}}(\Sigma_i) 
&\geq\liminf_{\rho\rightarrow\infty}\liminf_{i\rightarrow\infty} F_{(d_i-\frac{\rho}{d_i})y_i,1-\frac{1}{d_i^2}(d_i-\frac{\rho}{d_i})^2}(\Sigma_i)\\
&=\liminf_{\rho\rightarrow\infty}\liminf_{i\rightarrow\infty}F_{0,1}(\frac{d_i(\Sigma_i-d_iy_i)+\rho y_i}{\sqrt{2\rho-\rho^2/d_i^2}})\\
&=\liminf_{\rho\rightarrow\infty}F_{-\frac{\rho y}{2},\frac{\rho}{2}}(\Sigma_y)\\
&=\lambda_1.
\end{align*}
Here the first inequality follows from the fact that for a fixed $\rho$, when $d_i$ is sufficiently large, $(d_i-\frac{\rho}{d_i})\geq\rho_i$.
By Proposition \ref{prop:entropy upper bound}
\[
\limsup_{i\rightarrow\infty}F_{\rho_iy_i,1-\frac{\rho_i^2}{d_i^2}}(\Sigma_i)
\leq \liminf_{\rho\rightarrow\infty}F_{\rho y,1}(\Sigma)=\lambda_1.
\]
This completes the proof.
\end{proof}	

{\noindent\bf Step 2:}
{\em For any $\rho_i\rightarrow+\infty$ with $d_i(d_i-\rho_i)\rightarrow+\infty$, $\frac{\Sigma_i-\rho_iy_i}{\sqrt{1-\rho_i^2/d_i^2}}$ locally smoothly converges to $S^1(\sqrt2)\times\mathbb R_y$ up to a subsequence.

For $\rho_i\rightarrow+\infty$ with $d_i(d_i-\rho_i)\rightarrow \rho$, $\frac{\Sigma_i-\rho_iy_i}{\sqrt{1-\rho_i^2/d_i^2}}$ locally smoothly converges to $\frac{2\Sigma_y+\rho y}{\sqrt {2\rho}}$. }

\begin{remark}
Note that Theorem \ref{Theorem: blow up shrinker to get translator} gives that  $\frac{d_i}{2}(\Sigma_i-d_iy_i)$ locally smoothly converges to a translating soliton $\Sigma_y$ with translating direction $-y$. On the one hand, the blow-down sequence of the end of a Bowl soliton is a round cylinder. On the other hand, $\Sigma-\rho y$ also converges to a round cylinder. In the first case of Step 2, the sequence must converge to a round cylinder. 
\end{remark}

\begin{proof}[Proof of Step 2]
Define the Brakke flow for $t<0$,
\[ \mu_t^i=\mathcal H^2\lfloor\Big(\sqrt{\frac{1}{1-\rho_i^2/d_i^2}-1-t}\,\Sigma_i-\frac{\rho_iy_i}{\sqrt{1-\rho_i^2/d_i^2}}\Big).\]
By the compactness of Brakke flows, up to a subsequence, $\mu_t^i$ converges to a Brakke flow $\{\nu_t\}_{t<0}$ in the sense of Radon measures. Computing directly, we have
\begin{equation}
\mu_t^i=\mathcal H^2\lfloor\Big(\sqrt{-t}(\frac{\Sigma_i-\widetilde\rho_iy_i}{\sqrt{1-\widetilde\rho_i^2/d_i^2}})\Big),
\end{equation}	
where $\widetilde\rho_i=\frac{\rho_i}{\sqrt{1-(1+t)(1-\rho_i^2/d_i^2)}}$. Note that for any $t<0$ fixed, 
\[
\frac{\widetilde\rho_i}{d_i}=\frac{\rho_i/d_i}{\sqrt{1-(1+t)(1-\rho_i^2/d_i^2)}}
                            =\frac{\rho_i/d_i}{\sqrt{\rho_i^2/d_i^2-(1-\rho_i^2/d_i^2)t}}.
\]
Hence $\frac{\widetilde \rho_i}{d_i}<1$ for sufficiently large $i$. Moreover,
\begin{align*}
1-\frac{\widetilde \rho_i^2}{d_i^2}&=\frac{-t(1-\rho_i^2/d_i^2)}{\rho_i^2/d_i^2-(1-\rho_i^2/d_i^2)t}\\
                &>\frac{-t}{1-t}(1-\frac{\rho_i^2}{d_i^2}).
\end{align*}
Since $d_i^2-\rho_i^2=d_i^2-d_i\rho_i+d_i\rho_i-\rho_i^2\geq \rho_i(d_i-\rho_i+1)\geq\rho_i\rightarrow +\infty$, we have that $d_i^2-\widetilde \rho_i^2\rightarrow+\infty$. Consequently, $d_i(d_i-\widetilde \rho_i)\rightarrow+\infty$.

Now we can compute the $F$-functional as follows:
\begin{align*}
\int\Phi(x,t)d\nu_t&=\lim_{i\rightarrow\infty}\int\Phi(x,t)d\mu_t^i\\
                   &=\lim_{i\rightarrow\infty}\int_{\frac{\Sigma_i-\widetilde\rho_iy_i}{\sqrt{1-\widetilde\rho_i^2/d_i^2}}}\Phi(x,-1)dx\\
                   &=\lim_{i\rightarrow\infty}F_{\widetilde\rho_iy_i,1-\frac{\widetilde\rho_i^2}{d_i^2}}(\Sigma_i)\\
                   &=\lambda_1.
\end{align*} 

It remains to show that $\nu_{-1}$ supports on a smooth self-shrinker. Indeed, by monotonicity formulas, $\nu_{-1}$ is $F$-stationary, i.e. 
\[\overrightarrow{H}+\frac{S(x)^\perp\cdot x}{2}=0,\ \ \ \nu_{-1}-a.e.\  x.\]

Now we are going to prove that $\nu_{-1}$ splits off a line. For any $\tau\in\mathbb R$, set 
\[\rho'_i=\rho_i+\tau\sqrt{1-\frac{\rho_i^2}{d_i^2}} \text{ \ \  \ and\ \ \ } s_i=\frac{1-\rho'^2_i/d_i^2}{1-\rho_i^2/d_i^2}.\]
It follows that $d_i(d_i-\rho_i')\rightarrow+\infty$ and $s_i\rightarrow 1$ as $i\rightarrow\infty$ if $d_i(d_i-\rho_i)\rightarrow\infty$ and $\rho_i\rightarrow\infty$. Thus $\nu_{-1}-\tau y$, the limit of $\mu_{-1}^i-\tau y_i=\mathcal H^2\lfloor \sqrt{s_i}\frac{\Sigma_i-\rho'_iy_i}{\sqrt{1-\rho_i'^2/d_i^2}}$, is also $F$-stationary. This deduces that 
\[\overrightarrow H(x+y)+\frac{S(x+y)^\perp\cdot (x+y)}{2}=0, \nu_{-1}-a.e.\]
Therefore, $S(x)^\perp y=0$, which implies that $\nu_{-1}$ splits off a line. Note $F_{0,1}(\nu_{-1})=\lambda_1$. It follows that $\nu_{-1}$ is supporting on $S^1(\sqrt 2)\times \mathbb R_y$ with multiplicity $1$. Then by the local regularity theorem for Brakke flows, we conclude that $\frac{\Sigma_i-\rho_iy_i}{\sqrt{1-\rho_i^2/d_i^2}}$ locally smoothly converges to $S^1(\sqrt2)\times\mathbb R$ up to a subsequence.

The second statement follows from the fact that $\frac{d_i}{2}(\Sigma_i-d_iy_i)$ locally smoothly converges to a Bowl soliton $\Sigma_y$.
\end{proof}

{\noindent\bf Step 3:}
{\em For sufficiently large $i$, $\Sigma_i\cap \mathcal{C}_{U,R}\subseteq N_i$, where $N_i$ (see Figure \ref{fig:Ni}) is defined as
\[N_i=\bigcup_{\rho=0}^{d_i-10/d_i}B^{\rho y_i}_{10\sqrt{1-\rho ^2/d_i^2}}(\rho y_i)\bigcup\bigcup_{\rho=d_i-20/d_i}^{d_i}B^{\rho y_i}_{10\sqrt{1+d_i(d_i-\rho)}/d_i}(\rho y_i).\]}

%{\color{blue} Intuitively this region looks like a pencil, and $d_iy_i$ is the tip.}

\begin{figure}[h]
\begin{center}
\def\svgwidth{0.9\columnwidth}
  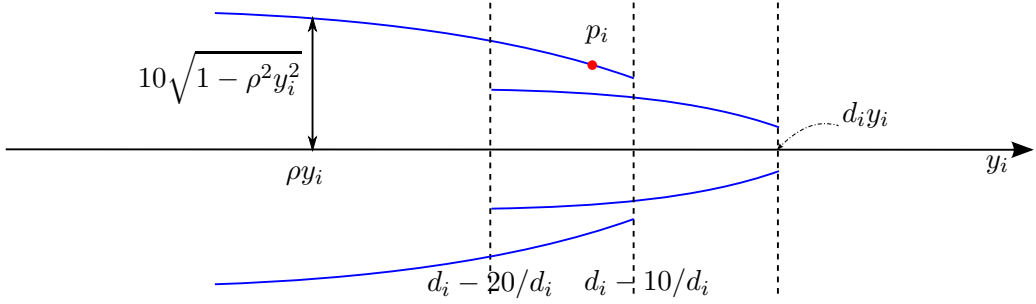
  \caption{Boundary of $N_i$.}
  \label{fig:Ni}
\end{center}
\end{figure}

\begin{proof}[Proof of Step 3]
Assume on the contrary that a subsequence of $\Sigma_i\cap \mathcal{C}_{U,R}$ is not contained in $N_i$. We still denote the subsequence by $\{\Sigma_i\}$. Then for each $i$, $\Sigma_i\cap \mathcal{C}_{U,R}\cap\partial N_i\not= \emptyset$. Take $p_i\in\Sigma_i\cap \mc C_{U,R}\cap\partial N_i$ and set $\rho_i=\langle p_i,y_i\rangle$. 

\medskip
If $\rho_i$ has a bounded subsequence, take a subsequence of $p_i$ (still denoted by $p_i$) such that $\lim\rho_i=\rho_\infty>R$. Then for large $i$, $p_i\in\partial B^{\rho_i y_i}_{10\sqrt{1-\rho_i^2/d_i^2}}(\rho_i y_i)$, which implies that 
\begin{equation}\label{eq:bounded rhoi}
\Bigg|\frac{p_i-\rho_iy_i}{\sqrt{1-\rho_i^2/d_i^2}}\Bigg|=10,
\end{equation}
Recall that $\Sigma_i-\rho_iy_i$ locally smoothly converges to $\Sigma-\rho_\infty y$. Then \eqref{eq:bounded rhoi} contradicts the fact that $\Sigma-\rho y$ is close to $S^1(\sqrt 2)\times \mb R_y$ near the origin for $\rho>R$..

\medskip
If $\rho_i\rightarrow\infty$ and $d_i(d_i-\rho_i)$ is unbounded, take a subsequence of $p_i$ (still denoted by $p_i$) such that $\lim d_i(d_i-\rho_i)=+\infty$. Then again $p_i\in\partial B^{\rho y_i}_{10\sqrt{1-\rho ^2/d_i^2}}(\rho y_i)$, which implies that
\[\Bigg|\frac{p_i-\rho_iy_i}{\sqrt{1-\rho_i^2/d_i^2}}\Bigg|=10.\]
This contradicts the fact that $\frac{\Sigma_i-\rho_iy_i}{\sqrt{1-\rho_i^2/d_i^2}}$ locally smoothly converges to $S^1(\sqrt 2)\times \mathbb R_y$.

\medskip
It remains to consider the case that $\rho_i\rightarrow\infty$ and $d_i(d_i-\rho_i)$ is bounded. Note that the boundedness of $d_i(d_i-\rho_i)$ implies that $d_i^2-\rho_i^2=(d_i+\rho_i)(d_i-\rho_i)$ is also bounded. Let $\tau$ be a limit of $d_i(d_i-\rho_i)$ (up to a subsequence). Then 
\[d_i(\Sigma_i-\rho_iy_i)=d_i(\Sigma_i-d_iy_i)+d_i(d_i-\rho_i)y_i\rightarrow 2\Sigma_y+\tau y,\]
where $\Sigma_y$ is the Bowl soliton with direction $-y$. Denote by $p$ the limit of $d_i(p_i-\rho_iy_i)$. Then 
\begin{equation}\label{eq:step3:p in plane}
p\in\{x\in\mathbb R^3:\langle x,y\rangle=0\}\cap (2\Sigma_y+\tau y).
\end{equation}

\begin{figure}[h]
\begin{center}
\def\svgwidth{0.65\columnwidth}
  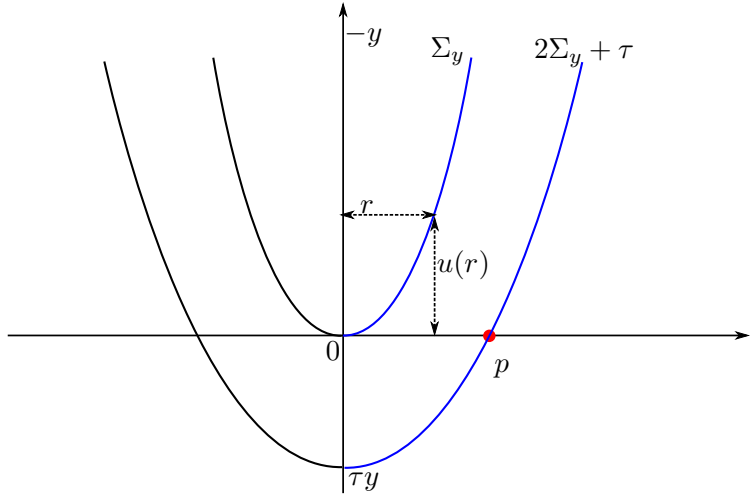
  \caption{Bowl soliton.}
  \label{fig:bowl}
\end{center}
\end{figure}

Let $u:[0,\infty)\rightarrow \mathbb R$ be the graph function of $\Sigma_y\cap(\mathbb R_y\times \mathbb R_p)$ (see Figure \ref{fig:bowl}). Then $u$ satisfies (see \cite{ADS19}*{(8.7)})
\[\frac{u''}{1+{u'}^2}+\frac{u'}{r}=\frac{1}{2},\]
and $u(0)=u'(0)=0$. Since the Bowl soliton is convex, $u''\geq 0$ everywhere. Hence $u''+u'/r\geq 1/2$, which implies that $u\geq r^2/8$. Then a direct computation and (\ref{eq:step3:p in plane}) give that 
\begin{equation}\label{eq:estimate of |p|}
|p|\leq 4\sqrt{\tau}.
\end{equation}

Now we are going to estimate $|p|$ by the choice of $p_i$ to get a contradiction. 
\begin{claim}\label{claim:distance has a lower bound}
For $\rho_i$ satisfying $lim_{i\rightarrow \infty}d_i(d_i-\rho_i)<\infty$, we always have 
\[|p_i-\rho_iy_i|\geq 10\sqrt{1+d_i(d_i-\rho_i)}/d_i.\]
\end{claim}
\begin{proof}
If $\rho_i\leq d_i-20/d_i$ for large $i$, then we always have 
\[|p_i-\rho_iy_i|=10\sqrt{1-\rho^2/d_i^2}.\]
Recall that the boundedness of $d_i(d_i-\rho_i)$ implies that $\rho_i>d_i-1$ for large $i$ and 
\begin{equation}\label{eq:comparison of two radii}
10\sqrt{1-\rho^2/d_i^2}\geq 10\sqrt{1+d_i(d_i-\rho)}/d_i
\end{equation}
is equivalent to $\rho^2-d_i\rho+1\leq 0$. Hence, the claim holds true in this case.

We now consider $\rho_i>d_i-20/d_i$. In this case, the claim follows from the definition of $N_i$.
\end{proof}

Claim \ref{claim:distance has a lower bound} gives that 
\begin{equation*}
|p|=\lim_{i\rightarrow\infty}d_i\cdot|p_i-\rho_i y_i|\geq \lim_{i\rightarrow\infty}10\sqrt{1+d_i(d_i-\rho)}=10\sqrt {1+\tau}.
\end{equation*}
which leads to a contradiction to (\ref{eq:estimate of |p|}).

Above all, we must have $(\Sigma_i\cap\mathcal{C}_{U,R})\subseteq N_i$.
\end{proof}

\medskip
Now we argue by contradiction to study the topology of $\Sigma_i\cap\mathcal{C}_{U,R}=\Sigma_i\cap N_i$ outside a small neighborhood of $d_iy_i$. We want to prove that the ends $\Sigma_i\cap\mathcal{C}_{U,R}$ outside a small neighborhood of $d_iy_i$ are topological cylinders. Let us first sketch the ideas. Roughly speaking, suppose that $r_iy_i$ are the bad points, in the sense that they prevent the parts of ends being cylinder, then
\begin{itemize}
\item if $r_i$ is bounded, it contradicts the fact that $\Sigma_i$ locally smoothly converges to $\Sigma$;
\item if $r_i$ is unbounded and $d_i(d_i-r_i)$ is unbounded, we apply Step 2 to get a contradiction;
\item if $r_i$ is unbounded and $d_i(d_i-r_i)$ is bounded, we use Theorem \ref{Theorem: blow up shrinker to get translator} to get a contradiction.
\end{itemize} 
The following are the precise arguments.

{\noindent\bf Step 4:}
{\em $\Sigma_i\cap\mathcal C_{U,R}\setminus \mathcal C_{U,d_i-20/d_i}$ is diffeomorphic to $\Sigma\cap\mathcal C_{U,R}\setminus\mathcal C_{U,d_i-20/d_i}$ .}
\begin{proof}[Proof of Step 4]
It suffices to show that there exists $I>0$ such that for any $i>I$ and $r_i\in(R,d_i-20/d_i)$, $\Sigma_i\cap\bigcup_{\rho=r_i}^{r_i+\sqrt{1-r_i^2/d_i^2}}B_{\sqrt{1-\rho^2/d_i^2}}^{\rho y_i}(\rho y_i)$ is diffeomorphic to $S^1\times [0,1]$. Set  
\[\Sigma_i(r_i):=\Sigma_i\cap\bigcup_{\rho=r_i}^{r_i+\sqrt{1-r_i^2/d_i^2}}B_{\sqrt{1-\rho^2/d_i^2}}^{\rho y_i}(\rho y_i).\]

We argue by contradiction. Assume the statement is not true. Then for each $i$, take $r_i\in(R,d_i-20/d_i)$ such that $\Sigma_i(r_i)$ is not diffeomorphic to $S^1\times [0,1]$. 

\medskip
If $r_i$ is bounded, let $r>0$ so that a subsequence of $r_i$ (still denoted by $r_i$) satisfies $\lim r_i=r$. Then $\Sigma_i(r_i)$ locally smoothly converges to $(\Sigma-r y)\cap B_2(0)\times [0,1]$, which leads to a contradiction.

\medskip
If $r_i$ is unbounded and $d_i(d_i-r_i)$ is unbounded, take a subsequence of $r_i$ (still denoted by $r_i$) such that $\lim r_i=+\infty$ and $\lim d_i(d_i-r_i)=+\infty$. Then by Step 2, $\frac{\Sigma_i-r_iy_i}{\sqrt{1-r_i^2/d_i^2}}$ locally smoothly converges to $S^1(\sqrt2)\times\mathbb R_y$. Therefore, $\frac{\Sigma_i(r_i)-r_iy_i}{\sqrt{1-r_i^2/d_i^2}}$ locally smoothly converges to $S^1(\sqrt{2})\times [0,1]$, which leads to a contradiction.

\medskip
If $r_i$ is unbounded but $d_i(d_i-r_i)$ is bounded, take a subsequence of $r_i$ (still denoted by $r_i$) such that $\lim r_i=+\infty$ and $\lim d_i(d_i-r_i)=\alpha$. Then by assumptions, $\alpha\geq 20$. Note that 
\[\lim_{i\rightarrow\infty}d_i\sqrt{1-r_i^2 /d_i^2}=\lim_{i\rightarrow\infty}\sqrt{(d_i+r_i)(d_i-r_i)}=\sqrt{2\alpha}.\]
Together with Theorem \ref{Theorem: blow up shrinker to get translator}, it follows that $\frac{\Sigma_i(r_i)-r_iy_i}{\sqrt{1-r_i^2/d_i^2}}$ locally smoothly converges to $\frac{2\Sigma_y+\alpha y}{\sqrt {2\alpha}}\cap(B_2(0)\times [0,1])$, which leads to a contradiction.
\end{proof}

\vspace{0.5em}
We now proceed the proof of Proposition \ref{prop:compact}. Step 2 suggests that $\Sigma_i\cap \mathcal{C}_{U,d_i-20/d_i}$ is diffeomorphic to a bowl soliton, i.e. a disk; Step 4 suggests that $\Sigma_i\cap \mathcal{C}_{U,R}\backslash \mathcal{C}_{U,d_i-20/d_i}$ is diffeomorphic to a cylinder. Step 3 suggests that $\Sigma_i\cap \mathcal{C}_{U,R}$ is just the union of these two parts.
Now the proposition follows from these steps.

\end{proof}

\begin{proof}[Proof of Theorem \ref{thm:shrinker converge}]
The desired results follow from Theorem \ref{Theorem: Asymptotic to a cone main theorem} and Proposition \ref{prop:cylinder end}, \ref{prop:cone} and \ref{prop:compact}. 
\end{proof}

\appendix

\bigskip
\section{Classification of self-shrinker ends with finite genus}\label{section:end classification}
In \cite{Wang16}, L. Wang classified properly embedded self-shrinker ends with finite topology. The key step is to prove the multiplicity of the tangent flow is $1$. Though L. Wang stated the result for finite topology, the argument works for all self-shrinkers with finite genus. We state it here and sketch the proof for completeness.
\begin{theorem}[\cite{Wang16}*{Theorem 1.1}]\label{thm:tangent flow multiplicity 1}
Let $M\in\mathcal M_\mathcal S(\Lambda,g)$ and $y\in S^2(1)$ such that \[\lim_{\rho\rightarrow\infty}F_{\rho y,1}(M)\geq 1.\] 
Then $\Sigma-\rho y$ locally smoothly converges to $\mathbb R_y\times \mathbb R$ or $\mathbb R_y\times S^1(\sqrt 2)$.
\end{theorem} 
\begin{proof}
In this proof, we will use the same notations as in \cite{Wang16}. 

The theorem follows from \cite{Wang16}*{\S 5} if we prove \cite{Wang16}*{Theorem 1.3} under the assumptions of finite genus instead of finite topology. Then there exists $\rho_0>0$ so that $M\setminus B_{\rho_0}(0)$ does not contain two simple closed curves with intersection number $1$. Now take a connected component of $M\setminus B_{\rho_0}(0)$, still denoted by $M$. Then there exists $\rho_1>0$ so that $\partial M\subset B_{\rho_1}(0)$ and $\partial M$ lie in a connected component of $M\cap B_{\rho_1}(0)$.

Note that \cite{Wang16}*{Proposition 4.1} holds true under the assumption of finite genus. Hence the desired results follow from \cite{Wang16}*{Theorem 1.3}. We now sketch the proof by pointing out the necessary modifications.

Without any changes, L. Wang's argument \cite{Wang16}*{Proof of Theorem 1.3} shows that given $i$ sufficiently large, for any $k\geq i$ and any connected component 
\[M^{i,k}\subset M\cap(I_y(\tau_k,\tau_{k+1})\times B^2_{R_i}),\]
the two sets
\[\partial M ^{i,k}\cap\{x \in\mathbb R^3: x\cdot y = \tau_k\}
\text{ and } \partial M^{i,k}\cap \{x\in\mathbb R^3: x\cdot y = \tau_{k+1}\}\]
are union of the same number of elements of $\{\gamma_{j}^{i,k}\}$ and $\{\gamma_j^{i,k+1}\}$, respectively. Hence the connected components of $M\cap B^3_{R_i}(\tau_iy)$ extend to infinity for all $i$ sufficiently large.

Then similar to \cite{Wang16}*{Proof of Theorem 1.3}, it suffices to show that, for sufficiently large $i$, the connected components of $M\cap B^3_{R_i}(\tau_iy)$ remain disjoint in the solid half-cylinder $I_y(\tau_i,+\infty)\times B^2_{R_i}$. 

Suppose not, then there exist $i>0$ and $k_l\rightarrow+\infty$ so that two elements of $\{\gamma^{i,k_l}_j\}_{j=1}^L$ lie in the boundary of a connected component of $M\cap (I_y(\tau_{k_l},\tau_{k_l+1})\times B_{R_i}^2)$. Hence there exist $k_0$ sufficiently large and a connected component of $M\cap(I_y(\tau_{k_0},+\infty)\times B^2_{R_i})$, denoted by $N^{i,k_0}$, which is connected at infinity, so that the boundary of a connected component of $N^{i,k_0}\cap (I_y(\tau_{k_0},\tau_{k_0+1})\times B_{R_i}^2)$ contains two of $\{\gamma^{i,k_0+1}_j\}_{j=1}^L$. Denote such a component by $M^{i,k_0}$. Then we can write
\[\partial M^{i,k_0}\cap \{x\in\mathbb R^3:x\cdot y=\tau_{k_0+1}\}=\bigcup_{l=1}^m\gamma_{j_l}^{i,k_0+1},\]
where $m\geq 2$. 

Then the following argument is similar to \cite{Wang16}*{Proof of Theorem 1.3}. Indeed, we can take $\{z_l\}_{l=1}^2,\{\gamma_l\}_{l=1}^2,\gamma$ like L. Wang did in \cite{Wang16}*{Proof of Theorem 1.3}. As $N^{i,k_0}$ is connected at infinity, there exist points $z_1'\in\gamma_1$ and $z_2'\in\gamma_2$ so that they can be jointed by a smooth embedded curve $\gamma'$ in $N^{i,k_0}$ outside a large enough ball so that $\gamma'\cap \gamma_{j_1}^{i,k_0+1}=\emptyset$. Denote the part of $\gamma$ connecting $z_1'$ and $z_2'$ that contains $\gamma_0$ by $\gamma''$. Thus, by a smoothing process, $\overline{\gamma'\cup\gamma''}$ can be made a simple closed smooth curve in $M$ that transversally intersects $\gamma_{j_1}^{i,k_0+1}$ only at $z_1$. Then from the curve $\gamma_{j_1}^{i,k_0+1}$ we will construct, in the following paragraphs, a simple closed curve that transversally intersects $\overline{\gamma'\cup\gamma''}$ only at one point. This leads to a contradiction.

If $\Sigma^{\tau_{k_0+1}}=\mathbb R_y\times \mathbb S^1$, then $\gamma^{i,k_0+1}_{j_1}$ is closed and we are done.

If $\Sigma^{\tau_{k_0+1}}=\mathbb R_y\times \mathbb R$, then by using the maximum principle for self-shrinkers \cite{Wang16}*{Lemma 4.2}, L. Wang has proved that there is a curve $\tilde{\gamma_l}$ in $M$ jointing $z_l$ to $\partial M$ for $l=1,2$ which is disjoint with $\gamma_0$. Note that $\partial M$ lies on a connected component of $M\cap B_{\rho_1}(0)$. Hence $z_1$ and $z_2$ can be jointed by a curve, denoted by $\gamma_{1,2}$, which is disjoint with $\gamma_0$ or $\overline{\gamma'\cup\gamma''}$. Thus, by a smooth process again, $\overline{\gamma_{1,2}\cup\gamma_0}$ can be made the desired simple closed smooth curve.
\end{proof}

Here is an application of Theorem \ref{thm:tangent flow multiplicity 1}.
\begin{lemma}\label{lem:finite ends}
Let $M$ be a self-shrinker with finite entropy and genus. Then $M$ has finitely many ends.
\end{lemma}
\begin{proof}
Assume on the contrary that $M$ has finite entropy and finite genus, but infinitely many ends. 

\vspace{0.5em}
{\noindent\em Case 1: There are infinitely many cylindrical ends.}

Suppose that $M$ has cylindrical ends with direction $y_i$ and $|y_i|=1$ . Assume that $y_i\rightarrow y$. Then 
\[F_{\rho y,1}(M)=\lim_{i\rightarrow\infty }F_{\rho y_i,1}(M)\geq \lim_{i\rightarrow\infty}\lim_{r\rightarrow\infty}F_{ry_i,1}(M)=\lambda_1.\] 

By \cite{Wang16}, $M-\rho y$ locally smoothly converges to $\mathbb R_y\times S^1(\sqrt 2)$ with multiplicity $1$. Thus for any $R>2$, there exists $L_0>0$ such that 
\begin{equation}\label{eq:empty cylinder}
M\cap \{x\in\R^3:\langle x, y\rangle \geq L_0, \dist(x,I_y(L_0,+\infty))\in(2,R)\}=\emptyset.
\end{equation}
Take $i$ large enough such that $L_0|y_i-y|<1$. Consider a smooth function $f\in C^\infty(0,+\infty)$ given by 
\[f(\rho)=\rho^2-\langle \rho y_i,y\rangle^2.\]

Note that $f(L_0)<1$ and $\lim_{\rho\rightarrow\infty}f(\rho)=+\infty$. Hence there exists $L_1>L_0$ such that $f(L_1)=R^2/4$. Together with (\ref{eq:empty cylinder}), we have $(M-L_1y_i)\cap B_{R/4}(0)=\emptyset$. This deduces that $F_{L_1y_i,1}(M)\leq C/R$, which contradicts 
\[ F_{L_1y_i,1}(M)\geq \lim_{\rho\rightarrow\infty}F_{\rho y_i,1}(M)=\lambda_1.\]

\vspace{0.5em}
{\noindent\em Case 2: There are infinitely many conical ends.}

Suppose that $M$ admits infinitely many conical ends with links $\{\gamma_i\}$. Take $y_i\in \gamma_i$ and let $y=\lim_{i\to\infty} y_i$. Then
\[F_{\rho y,1}(M)=\lim_{i\rightarrow\infty }F_{\rho y_i,1}(M)\geq \lim_{i\rightarrow\infty}\lim_{r\rightarrow\infty}F_{ry_i,1}(M)=1.\] 

By \cite{Wang16}, $\Sigma-\rho y$ locally smoothly converges to $\mathbb R_y\times S^1(\sqrt 2)$ or $\mathbb R_y\times \mathbb R$. By the similar argument as in Case 1, the limit surface can not be a cylinder. Then it remains to rule out the case that $M-\rho y$ locally smoothly converges to $\mathbb R_y\times \mathbb R$. Without loss of generality, we assume that $\dist(y,\gamma_i)=\dist(y,y_i)$, where $\dist(\cdot,\cdot)$ is the distance function in $\mathbb R^3$.

Take $R\gg 1$. Then applying local regularity theorem \cite{Whi02}*{Theorem 3.1}, we have that:
\begin{lemma}\label{lem:app:as graphs]}
There exists $\epsilon>0$ so that, for any $x\in S^2(1)$, $\alpha>0$ with
\begin{equation}\label{eq:app:condition}
\text{ $\lim_{\rho\rightarrow\infty}F_{\rho x,1}(M)=1$ and $1\leq F_{\alpha x,1}(M)<1+\epsilon$,}
\end{equation}
then $M\cap B_{3R}(\alpha x)$ can be written as a graph, with $C^2$ norm $\ll 1$, over a hyperplane $P$ containing $\mathbb R_x$. Particularly, 
\[d_H(P\cap B_{3R}(\alpha x),M\cap B_{3R}(\alpha x))\ll 1,\]
where $d_H$ is the Hausdorff distance.
\end{lemma}
\begin{proof}[Proof of Lemma \ref{lem:app:as graphs]}]
The proof here is similar to Step 1 and 2 in Proposition \ref{prop:cylinder end}, and we leave it to readers. 
\end{proof}

\begin{claim}\label{claim:app:rho1 exists}
There exists $\rho_1>0$ so that $F_{x,1}(M)<1+\epsilon$ for all $x\in B_R(\rho y)$ and $\rho>\rho_1$.
\end{claim}
\begin{proof}[Proof of Claim \ref{claim:app:rho1 exists}]
Since $\lim_{\rho\rightarrow\infty}F_{\rho y,1}(M)=1$, then we can take $\rho_0$ and $r_0$ so that $F_{x,1}(M)<1+\epsilon$ for $x\in B_{r_0}(\rho_0y)$. Let $\rho_1=\rho_0\cdot R/r_0$. Then monotonicity of $F_{\rho x,1}(M)$ as a function of $\rho$ suggests that such $\rho_1$ is the desired constant.
\end{proof}

Now take $k$ large enough so that $\rho_1\dist(y_k,y)<1$. Set $\rho_2=10/\dist(y_k,y)$. It follows that $\dist(\rho_2y,\mathcal C_{\gamma_k})>9$. Note that $x=y$ and $\alpha=\rho_2$ satisfy (\ref{eq:app:condition}). By Lemma \ref{lem:app:as graphs]}, there exists a hyperplane $P$ containing $\mathbb R_y$ so that $M\cap B_{3R}(\rho_2y)$ can be written as a graph over $P$ with
\[d_H(P\cap B_{3R}(\rho_2y),M\cap B_{3R}(\rho_2y))\ll1.\]
We now prove that:
\begin{claim}\label{claim:gammak closes to P}
$\mc C_{\gamma_k}$ closes to $P$ in $B_{R}(\rho_2y)$.
\end{claim}
\begin{proof}[Proof of Claim \ref{claim:gammak closes to P}]
Since $\rho_2y_k\in B_R(\rho_2y)$, then $\mathcal C_{\gamma_k}\cap B_R(\rho_2y)\neq \emptyset$. By Claim \ref{claim:app:rho1 exists}, for all $z\in \mathcal C_{\gamma_k}\cap B_R(\rho_2y)$, we have that $x=z/|z|$ and $\alpha=|z|$ satisfy (\ref{eq:app:condition}). By Lemma \ref{lem:app:as graphs]}, $\dist (z,M\cap B_{3R}(z))\ll1$ for all $z\in\mathcal C_{\gamma_k}\cap B_R(\rho_2y)$. By triangle inequalities,
\begin{align*}
\dist(z,P)&\leq \dist(z,M\cap B_R(z))+d_H(M\cap B_R(z),P\cap B_R(z))\\
&\leq\dist(z,M\cap B_R(z))+d_H(M\cap B_{3R}(\rho_2y),P\cap B_{3R}(\rho_2y))\ll 1.
\end{align*}
\end{proof}

\begin{figure}[h]
\begin{center}
\def\svgwidth{0.6\columnwidth}
  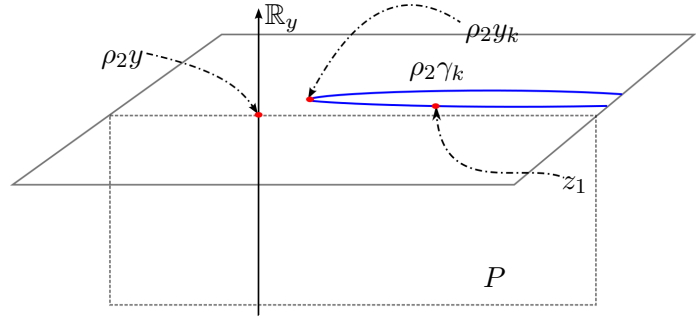
  \caption{$\rho_2\gamma_k$ closes to a line.}
  \label{fig:OverP}
\end{center}
\end{figure}

There are two cases, depending on whether $\rho_2\gamma_k$ is completely contained in the ball $B_R(\rho_2y)$. Let us first consider the case that $\rho_2\gamma_k$ is not completely contained in $B_R(\rho_2y)$. Let $\gamma'$ be the connected component of $\rho_2\gamma_k\cap B_R(\rho_2y)$ containing $\rho_2y_k$. Then by Claim \ref{claim:gammak closes to P} and the fact that $\rho_2y_k$ is the closest point of $\gamma'$ to $\rho_2y$, $\gamma'$ must be a curve close to half of $P\setminus \mb R_y$; see Figure \ref{fig:OverP}. Denote by $\gamma'_1$ and $\gamma_2'$ the two connected components of $\gamma'\setminus\{ \rho_2y_k\}$. Taking $z_1\in\rho_2\gamma_k$ so that $|z_1-\rho_2y_k|=R/2$, then we have 
\begin{align*}
F_{0,1}(\mc C_{\gamma_k}-z_1)&\geq F_{0,1}( B_{R/4}(z_1)\cap\mc C_{\gamma_1'}-z_1)+F_{0,1}( B_{R/4}(z_1)\cap\mc C_{\gamma_2'}-z_1)\\
&\geq 1-C/R+1-C/R>3/2.
\end{align*}
Therefore,
\[\lim_{\rho\rightarrow\infty}F_{\rho z_1/\rho_2,1+\rho^2/\rho^2_2}(M)=F_{0,1}(\lim_{\rho\rightarrow\infty}\rho^{-1}M-z_1)>3/2,\]
which implies that
\[
F_{z_1,2}(M)=F_{\rho_2 z_1/\rho_2,1+\rho_2^2/\rho^2_2}(M)\geq \lim_{\rho\rightarrow\infty}F_{\rho z_1/\rho_2,1+\rho^2/\rho^2_2}(M)
>3/2.
\]
This gives a contradiction.

To complete the proof, it suffices to consider the case that $\rho_2\gamma_k$ is contained in $B_R(\rho_2y)$. Let $x_1,x_2\in\rho_2\gamma_k$ so that $|x_1-x_2|=\mathrm{diam}(\rho_2\gamma_k)$ (see Figure \ref{fig:OverP1}). Set $\rho_3=2R/|x_1-x_2|$. Then $x=x_1/|x_1|$ and $\alpha=\rho_3|x_1|$ satisfy (\ref{eq:app:condition}). Hence by Lemma \ref{lem:app:as graphs]}, there exists a hyperplane $P_1$ containing $\mathbb R_{x_1}$ so that $M\cap B_{3R}(\rho_3x_1)$ can be written as a graph over $P_1$ with 
\[d_H(P_1\cap B_{3R}(\rho_3x_1),M\cap B_{3R}(\rho_3x_1))\ll1.\]
\begin{figure}[h]
\begin{center}
\def\svgwidth{0.6\columnwidth}
  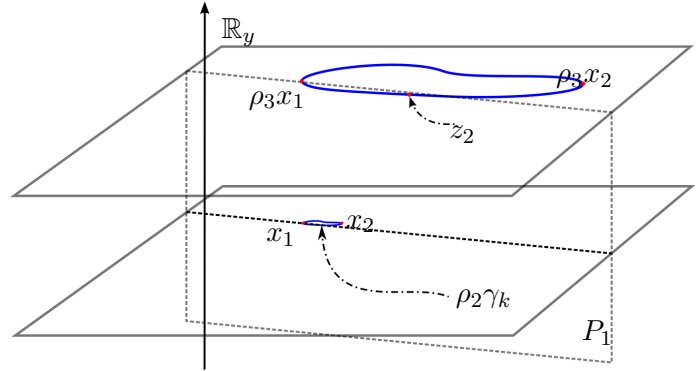
  \caption{$\rho_3\gamma_k$ closes to a line.}
  \label{fig:OverP1}
\end{center}
\end{figure}
\begin{claim}
$\mc C_{\gamma_k}$ closes to $P_1$ in $B_R(\rho_3x_1)$.
\end{claim}
\begin{proof}[Proof of Claim \ref{claim:gammak closes to P}]
Since $\rho_2y_k\in B_R(\rho_2y)$, then by Claim \ref{claim:app:rho1 exists}, 
$F_{z,1}(M)<1+\epsilon$ for all $z\in \rho_2\gamma_k$. By monotonicity formulas, for any $z\in \rho_3\rho_2\gamma_k$, $x=z/|z|$ and $\alpha=|z|$ satisfy \eqref{eq:app:condition}, i.e.
\[1\leq F_{z,1}(M)<1+\epsilon.\]
Then Lemma \ref{lem:app:as graphs]} gives that $\dist (z,M\cap B_{3R}(z))\ll1$ for all $z\in\rho_3\rho_2\gamma_k$. By triangle inequalities,
\begin{align*}
\dist(z,P_1)&\leq \dist(z,M\cap B_R(z))+d_H(M\cap B_R(z),P_1\cap B_R(z))\\
&\leq\dist(z,M\cap B_R(z))+d_H(M\cap B_{3R}(\rho_3x_1),P_1\cap B_{3R}(\rho_3x_1))\ll 1.
\end{align*}
\end{proof}

To proceed the argument, we take $z_2\in \rho_3\rho_2\gamma_k$ so that $|z_2-\rho_3x_1|=R$. Denote by $\gamma''_1$ and $\gamma_2''$ the two components of $\rho_2\rho_3\gamma_k\setminus\{\rho_3x_1,\rho_3x_2\}$. Then we have 
\begin{align*}
F_{0,1}(\mc C_{\gamma_k}-z_2)&\geq F_{0,1}( B_{R}(z_2)\cap\mc C_{\gamma_1''}-z_2)+F_{0,1}( B_{R}(z_2)\cap\mc C_{\gamma_2''}-z_2)\\
&\geq 1-C/R+1-C/R>3/2.
\end{align*}
Therefore,
\[\lim_{\rho\rightarrow\infty}F_{\rho z_2/\rho_2\rho_3,1+\rho^2/\rho^2_2\rho_3^2}(M)=F_{0,1}(\mc C_{\gamma_k}-z_2)>3/2,\]
which implies that
\[
F_{z_2,2}(M)=F_{\rho_2\rho_3 z_2/\rho_2\rho_3,1+\rho_2^2\rho_3^2/\rho^2_2\rho_3^2} (M)\geq \lim_{\rho\rightarrow\infty}F_{\rho z_2/\rho_2\rho_3,1+\rho^2/\rho^2_2\rho_3^2}(M)
>3/2.
\]
This gives a contradiction.
\end{proof}

The following linear bound of $|A|$ is proven by Song \cite{Song14}. Here we give a uniform bound for $M\in \mathcal M_\mathcal S(\Lambda,g)$, which is used in Theorem \ref{Theorem: blow up shrinker to get translator}.
\begin{lemma}[cf. \citelist{\cite{Wang16}*{Appendix A.1}\cite{Song14}*{Theorem 19}}]\label{lem:A linear bounded}
There exists a constant $C=C(\Lambda, g)$ such that for any self-shrinker $M\in\mathcal M_\mathcal S(\Lambda ,g)$, 
\[|A|(x)\leq C(|x|+1).\]
\end{lemma}
\begin{proof}[Proof of Lemma \ref{lem:A linear bounded}]
Assume on the contrary that there exist a sequence of self-shrinkers $\{\Sigma_i\}$ with entropy bound $\Lambda$ and genus bound $g$, and $x_k\in \Sigma_k$ such that 
\[|A|(x_k)\geq k(|x_k|+1).\] 
By Colding-Minicozzi \cite{CM12_2}, $\Sigma_k$ locally smoothly converges to a self-shrinker, denoted by $\Sigma$, with $\lambda(\Sigma)\leq \Lambda$ and $g(\Sigma)\leq g$. Without loss of generality, $\{x_k\}$ is assumed to be unbounded. Take $y_i\in B_{1/|x_i|}(x_i)$ such that
\[(|x_i|^{-1}-|y_i-x_i|)|A(y_i)|=\sup_{y\in B_{1/|x_i|}(x_i)}(|x_i|^{-1}-|y-x_i|)|A(y)|.\]
Set $r_i=(|x_i|^{-1}-|y_i-x_i|)/2$. Then $|A(y_k)|r_k\geq k/2$ and for $y\in B_{r_k}(y_k)\cap \Sigma_k$,
\[2r_k|A(y_k)|\geq (|x_k|^{-1}-|y-x_k|)|A(y)|\geq r_k|A(y)|.\]
Thus for $y\in B_{k/2}(0)\cap |A(y_k)|(\Sigma_k-y_k)$, $|A(y)|\leq 2$ and $|H(y)|\leq 1/k$. Let $k\rightarrow\infty$, $|A(y_k)|(\Sigma_k-y_k)$ locally smoothly converges to a minimal surface $\Gamma$. 

Let $y$ be a limit point of $y_k/|y_k|$. Then by Proposition \ref{prop:entropy upper bound}, for any $(z,t)\in\mathbb R^3\times (0,+\infty)$,
\[
F_{z,t^2}(\Gamma)\leq \lim_{\rho\rightarrow\infty}F_{\rho y,1} (\Sigma)\leq \lambda_1,           
\]
which implies that $\Gamma$ has only one end. Thus $\Sigma$ is a plane, which contradicts $|A_{\Gamma}(0)|=1$.
\end{proof}

\section{Translating solitons with low entropy}\label{section:Translating solitons with low entropy}
For completeness of this paper, we sketch the proof of classification of translating solitons with low entropy in $\mathbb R^3$. The argument is based on the proof by Choi-Haslhofer-Hershkovits \cite{CHH18}*{Theorem 1.2} and Hershkovits \cite{Her18}*{Theorem 3}. We refer to \cite{Gua16} the computation of the entropy of translating solitons.

\begin{theorem}\label{thm:translator low entropy}
Let $M\subseteq \mathbb R^3$ be a translating soliton with entropy less than $2$. Then $M$ is a plane or a Bowl soliton.
\end{theorem}
\begin{proof}
Let $y$ be the translating vector field of $M$, i.e.
\[H=-\langle y,\n\rangle.\]
For each $r>0$, define the blow-down Brakke flow by
\[\mu_t^r=\mathcal H^2\lfloor\Big(\frac{M-ry}{\sqrt{1+r}}+\sqrt{1+r}(1+t)y\Big).\]
By the compactness of Brakke flows (Lemma \ref{lem:brakke compactness}), as $r\rightarrow+\infty$, there exist a sequence $r_i\rightarrow\infty$ and a Brakke flow $\{\nu_t\}$ such that $\mu_t^{r_i}\rightharpoonup\nu_t$ in the sense of Radon measures for all $t<0$. 

We are going to show that $\nu_{-1}$ is $F$-stationary. Indeed, by computing directly,
\[\mu_t^r=\mathcal H^2\lfloor\Big(\sqrt{-t}\big(\frac{M-\widetilde ry}{\sqrt{1+\widetilde r}}\big)\Big),\]
where $\widetilde r=r-(1+t)(1+r)$. Note that for $t<0$ fixed, $\widetilde r\rightarrow+\infty$ as $r\rightarrow+\infty$.

Note that $M$ is a translating soliton with direction $y$. Then by the monotonicity of $F$-functional, $F_{ry,1+r}(M)$ is monotonically increasing (see \cite{CM12_1}*{(1.9)}). Denote by $\lambda$ the limiting value. Then $\lambda<2$ by the assumption of entropy bound.

Now computing the $F$-functional for the limiting flow gives
\[
\int\Phi(x,t)d\nu_t=\lim_{r\rightarrow\infty}\int\Phi(x,t)d\mu_t^r=\lim_{r\rightarrow+\infty}F_{\widetilde ry,1+\widetilde r}(\Sigma)=\lambda.\]

By the monotonicity formulas (see Lemma \ref{Brakke monotonicity}), $\nu_{-1}$ is $F$-stationary.

Now we can show that $\nu_{-1}$ is supporting on a smooth self-shrinker. Namely, by the same argument, for all $\tau\in\mathbb R$, the related varifold of the limit of $\mu_{-1}^{r_i+\tau\sqrt{1+r_i}}$ is also $F$-stationary. Recall that 
\[\mu_{-1}^{r_i+\tau\sqrt{1+r_i}}=\mathcal H^2\lfloor\sqrt{s_i}(\frac{M-r_iy}{\sqrt{1+r_i}}-\tau y),\]
where $s_i=\frac{\sqrt{1+r_i}}{\sqrt{1+r_i}+\tau}$. Note that $s_i\rightarrow 1$ as $r_i\rightarrow+\infty$. We conclude that $\nu_{-1}-\tau y$ is also stationary. This deduces that $y^\perp=0$. It follows that $\nu_{-1}$ splits off a line. 

Now $\nu_{-1}$ is a $F$-stationary cycle with entropy less than $2$. It follows that $\nu_{-1}$ is supporting on a plane or a self-shrinker cylinder.

In the first case, $M$ is a plane by Brakke local regularity or the argument in \cite{BW18}*{Proposition 3.2}. In the second case, $M$ is a Bowl soliton by Hershkovits \cite{Her18}*{Theorem 3}. 

For completeness of this paper, we sketch the proof given by Hershkovits \cite{Her18}. Firstly, using the Neck Improvement Theorem \cite{BC17}*{Theorem 4.4}, the end can be written as a graph (see \cite{Her18}*{\S 2} for details) of a function $g$ over $\mathbb R^2\setminus B_R(0)$ satisfying 
\[g=\frac{1}{2}(x_1^2+x_2^2)-\frac{1}{2}\log(x_1^2+x_2^2)+O(\frac{1}{\sqrt{x_1^2+x_2^2}}).\] 
Then the claim follows from \cite{MHS15}*{Theorem A} once we prove $M$ has only one end. 

We now show that $M$ admits only one end. We first claim that $|A|$ is bounded by Choi-Schoen's estimates \cite{CS85}.
\begin{claim}\label{claim:|A| bound}
There exists $C=C(M)$ such that $|A|(x)\leq C$ for any $x\in M$, where $A$ is the second fundamental form of $M$.
\end{claim}
\begin{proof}[Proof of Claim \ref{claim:|A| bound}]
Assume on the contrary that $x_i\in M$ satisfying $|A|(x_i)>i$. Take $q_i\in M$ such that 
\[\max_{x\in B_1(p)\cap M}(1-|x-x_i|)|A|(x)=(1-|z_i-x_i|)|A|(z_i),\]
where $B_r(p)$ is the ball in $\mathbb R^3$.

Set $r_i=(1-|z_i-x_i|)/2$. Then $r_i|A|(z_i)\rightarrow\infty$ and for any $x\in B_{r_i}(z_i)$, we have $|A|(x)\leq 2|A|(z_i)$. Hence $|A|(z_i)(M-z_i)$ locally smoothly converges to a hypersurface $\Gamma\subseteq \mathbb R^3$. Recall $M$ is a translating soliton and then $|H|\leq 1$. It follows that $\Gamma$ is a minimal surface. By Proposition \ref{prop:entropy upper bound}, the entropy of $\Gamma$ is less than $2$, which implies that $\Gamma$ is a plane. This contradicts the fact of $|A_\Gamma|(0)=1$. 

Thus we have proved Claim \ref{claim:|A| bound}.
\end{proof}

As a result of claim \ref{claim:|A| bound}, there exists a constant $\delta>0$ such that $F_{x,1}(M)\geq \delta$ for any $x\in M$.

To proceed our argument, assume on the contrary that $p_i=(x_{1,i},x_{2,i},y_i)\rightarrow\infty$ lies in $M$ but not in the end given by the graph of $g$. Recall that $\frac{M-ry}{\sqrt{1+r}}$ locally smoothly converges to $S^1(\sqrt 2)\times \mathbb R_y$, which yields that for any $k>0$ and large $i$, $y_i<\frac{1}{2k}(x_{1,i}^2+x_{2,i}^2)$. 
\begin{claim}\label{claim:parabolic neighborhood}
There exists $K>0$ such that $Ky_i\leq -(x_{1.i}^2+x_{2,i}^2)$ for all large $i$.
\end{claim}
\begin{proof}[Proof of Claim \ref{claim:parabolic neighborhood}]
Assume on the contrary that for any fixed $k>0$, we have $ky_i\ge -(x_{1,i}^2+x_{2,i}^2)$ for large $i$. Take $s_{i,k}>0$ such that $p_i+s_{i,k}y\in \{ky=x_{1}^2+x_{2}^2\}$ and set $r_i=(x_{1,i}^2+x_{2,i}^2)/k$ (see Figure \ref{fig:One_end}). 

\begin{figure}[h]
\begin{center}
\def\svgwidth{0.6\columnwidth}
  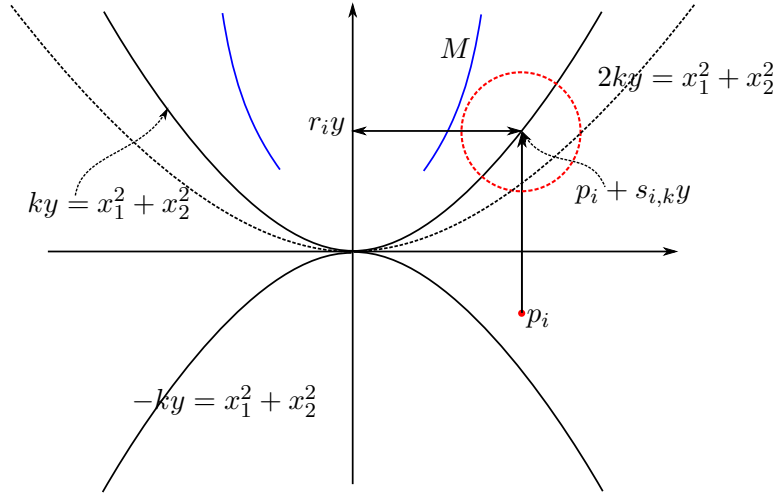
  \caption{$p_i\rightarrow-\infty$.}
  \label{fig:One_end}
\end{center}
\end{figure}

Because $\frac{M-r_iy}{\sqrt{1+r_i}}$ locally smoothly converges to $S^1(\sqrt 2)\times \mb R_y$, we have $B_{\sqrt{x_{1,i}^2+x_{2,i}^2}/2}(p_i+s_{i,k}y)\cap M=\emptyset$. Note that $y_i<\frac{1}{2k}(x_{1,i}^2+x_{2,i}^2)$. Clearly,
\[x_{1,i}^2+x_{2,i}^2\leq 2ks_{i,k}\leq 4(x_{1,i}^2+x_{2,i}^2).\]
Therefore,
\[B_{\sqrt k/4}(0)\cap \frac{M-p_i-s_{i,k}y}{\sqrt{1+s_{i,k}}}=\emptyset,\]
which implies that 
 \[ \lim_{k\rightarrow\infty}\lim_{i\rightarrow\infty}F_{s_{i,k}y,1+s_{i,k}}(M-p_i)=0. \]
 
Note that $M-p_i$ is also a translating soliton with direction $y$. Hence $F_{sy,1+s}(M-p_i)$ is monotonically increasing (see \cite{CM12_1}*{(1.9)}). This deduces that 
\[\delta\leq \liminf_{i\rightarrow\infty}F_{0,1}(M-p_i)\leq \lim_{k\rightarrow\infty}\lim_{i\rightarrow\infty}F_{s_{i,k}y,1+s_{i,k}}(M-p_i)=0,\]
which leads to a contradiction.
\end{proof}

Now set $p_i^0=(x_{1,i},x_{2,i},0)$. By Claim \ref{claim:parabolic neighborhood}, $Ky_i<-|p_i^0|^2$. In particular, $|y_i|\rightarrow\infty$ and $p_i^0/\sqrt {1+|y_i|}$ is bounded. Note that $\frac{M}{\sqrt{1+r}}=\mu_0^r$ vanishes as $r\rightarrow\infty$. Hence \[\lim_{i\rightarrow\infty}F_{0,1+|y_i|}(M-p_i^0)=0,\]
;see also Claim \ref{claim:parabolic neighborhood}). On the other hand,
\[F_{0,1}(M-p_i)\leq F_{|y_i|y,1+|y_i|}(M-p_i)=F_{0,1+|y_i|}(M-p_i^0),\]
which leads to a contradiction.

Then we finish the proof.
\end{proof}

\bibliographystyle{amsalpha}
\bibliography{self-shrinker}
\end{document}

%% file: cylinder.eps_tex
%% Creator: Inkscape inkscape 0.48.4, www.inkscape.org
%% PDF/EPS/PS + LaTeX output extension by Johan Engelen, 2010
%% Accompanies image file 'cylinder.eps' (pdf, eps, ps)
%%
%% To include the image in your LaTeX document, write
%%   \input{<filename>.pdf_tex}
%%  instead of
%%   \includegraphics{<filename>.pdf}
%% To scale the image, write
%%   \def\svgwidth{<desired width>}
%%   \input{<filename>.pdf_tex}
%%  instead of
%%   \includegraphics[width=<desired width>]{<filename>.pdf}
%%
%% Images with a different path to the parent latex file can
%% be accessed with the `import' package (which may need to be
%% installed) using
%%   \usepackage{import}
%% in the preamble, and then including the image with
%%   \import{<path to file>}{<filename>.pdf_tex}
%% Alternatively, one can specify
%%   \graphicspath{{<path to file>/}}
%% 
%% For more information, please see info/svg-inkscape on CTAN:
%%   http://tug.ctan.org/tex-archive/info/svg-inkscape
%%
\begingroup%
  \makeatletter%
  \providecommand\color[2][]{%
    \errmessage{(Inkscape) Color is used for the text in Inkscape, but the package 'color.sty' is not loaded}%
    \renewcommand\color[2][]{}%
  }%
  \providecommand\transparent[1]{%
    \errmessage{(Inkscape) Transparency is used (non-zero) for the text in Inkscape, but the package 'transparent.sty' is not loaded}%
    \renewcommand\transparent[1]{}%
  }%
  \providecommand\rotatebox[2]{#2}%
  \ifx\svgwidth\undefined%
    \setlength{\unitlength}{319.25270842bp}%
    \ifx\svgscale\undefined%
      \relax%
    \else%
      \setlength{\unitlength}{\unitlength * \real{\svgscale}}%
    \fi%
  \else%
    \setlength{\unitlength}{\svgwidth}%
  \fi%
  \global\let\svgwidth\undefined%
  \global\let\svgscale\undefined%
  \makeatother%
  \begin{picture}(1,0.14275889)%
    \put(0,0){\includegraphics[width=\unitlength]{cylinder.eps}}%
    \put(0.2709616,0.09203921){\color[rgb]{0,0,0}\makebox(0,0)[lb]{\smash{$2$}}}%
    \put(0.92848404,0.04286508){\color[rgb]{0,0,0}\makebox(0,0)[lb]{\smash{$y_i$}}}%
    \put(0.23374938,0.13556539){\color[rgb]{0,0,0}\makebox(0,0)[lb]{\smash{$p_i$}}}%
    \put(0.24504514,0.04848581){\color[rgb]{0,0,0}\makebox(0,0)[lb]{\smash{$\rho_i y_i$}}}%
    \put(0.51407795,0.12064391){\color[rgb]{0,0,0}\makebox(0,0)[lb]{\smash{$\partial N_i$}}}%
  \end{picture}%
\endgroup%

%% file: small_cones.eps_tex
%% Creator: Inkscape inkscape 0.48.4, www.inkscape.org
%% PDF/EPS/PS + LaTeX output extension by Johan Engelen, 2010
%% Accompanies image file 'small_cones.eps' (pdf, eps, ps)
%%
%% To include the image in your LaTeX document, write
%%   \input{<filename>.pdf_tex}
%%  instead of
%%   \includegraphics{<filename>.pdf}
%% To scale the image, write
%%   \def\svgwidth{<desired width>}
%%   \input{<filename>.pdf_tex}
%%  instead of
%%   \includegraphics[width=<desired width>]{<filename>.pdf}
%%
%% Images with a different path to the parent latex file can
%% be accessed with the `import' package (which may need to be
%% installed) using
%%   \usepackage{import}
%% in the preamble, and then including the image with
%%   \import{<path to file>}{<filename>.pdf_tex}
%% Alternatively, one can specify
%%   \graphicspath{{<path to file>/}}
%% 
%% For more information, please see info/svg-inkscape on CTAN:
%%   http://tug.ctan.org/tex-archive/info/svg-inkscape
%%
\begingroup%
  \makeatletter%
  \providecommand\color[2][]{%
    \errmessage{(Inkscape) Color is used for the text in Inkscape, but the package 'color.sty' is not loaded}%
    \renewcommand\color[2][]{}%
  }%
  \providecommand\transparent[1]{%
    \errmessage{(Inkscape) Transparency is used (non-zero) for the text in Inkscape, but the package 'transparent.sty' is not loaded}%
    \renewcommand\transparent[1]{}%
  }%
  \providecommand\rotatebox[2]{#2}%
  \ifx\svgwidth\undefined%
    \setlength{\unitlength}{433.77177604bp}%
    \ifx\svgscale\undefined%
      \relax%
    \else%
      \setlength{\unitlength}{\unitlength * \real{\svgscale}}%
    \fi%
  \else%
    \setlength{\unitlength}{\svgwidth}%
  \fi%
  \global\let\svgwidth\undefined%
  \global\let\svgscale\undefined%
  \makeatother%
  \begin{picture}(1,0.30240557)%
    \put(0,0){\includegraphics[width=\unitlength]{small_cones.eps}}%
    \put(0.78301004,0.27274567){\color[rgb]{0,0,0}\makebox(0,0)[lb]{\smash{$p_i=\rho_iz_i$}}}%
    \put(0.19989638,0.19575688){\color[rgb]{0,0,0}\makebox(0,0)[lb]{\smash{$k\beta_i^{-1}z_i$}}}%
    \put(0.303198,0.16514915){\color[rgb]{0,0,0}\makebox(0,0)[lb]{\smash{$>2k$}}}%
    \put(0.77587602,0.2122645){\color[rgb]{0,0,0}\rotatebox{-0.29273392}{\makebox(0,0)[lb]{\smash{$2\sqrt{1+\beta_i^2\rho_i^2}>2\beta_i\rho_i$}}}}%
    \put(0.94486239,0.12420114){\color[rgb]{0,0,0}\makebox(0,0)[lb]{\smash{$y_i$}}}%
    \put(0.39115927,0.06289355){\color[rgb]{0,0,0}\makebox(0,0)[lb]{\smash{$\partial N_i$}}}%
    \put(0.75251648,0.12657472){\color[rgb]{0,0,0}\makebox(0,0)[lb]{\smash{$\rho_iy_i$}}}%
  \end{picture}%
\endgroup%

%% file: NI.eps_tex
%% Creator: Inkscape inkscape 0.48.4, www.inkscape.org
%% PDF/EPS/PS + LaTeX output extension by Johan Engelen, 2010
%% Accompanies image file 'NI.eps' (pdf, eps, ps)
%%
%% To include the image in your LaTeX document, write
%%   \input{<filename>.pdf_tex}
%%  instead of
%%   \includegraphics{<filename>.pdf}
%% To scale the image, write
%%   \def\svgwidth{<desired width>}
%%   \input{<filename>.pdf_tex}
%%  instead of
%%   \includegraphics[width=<desired width>]{<filename>.pdf}
%%
%% Images with a different path to the parent latex file can
%% be accessed with the `import' package (which may need to be
%% installed) using
%%   \usepackage{import}
%% in the preamble, and then including the image with
%%   \import{<path to file>}{<filename>.pdf_tex}
%% Alternatively, one can specify
%%   \graphicspath{{<path to file>/}}
%% 
%% For more information, please see info/svg-inkscape on CTAN:
%%   http://tug.ctan.org/tex-archive/info/svg-inkscape
%%
\begingroup%
  \makeatletter%
  \providecommand\color[2][]{%
    \errmessage{(Inkscape) Color is used for the text in Inkscape, but the package 'color.sty' is not loaded}%
    \renewcommand\color[2][]{}%
  }%
  \providecommand\transparent[1]{%
    \errmessage{(Inkscape) Transparency is used (non-zero) for the text in Inkscape, but the package 'transparent.sty' is not loaded}%
    \renewcommand\transparent[1]{}%
  }%
  \providecommand\rotatebox[2]{#2}%
  \ifx\svgwidth\undefined%
    \setlength{\unitlength}{432.21918797bp}%
    \ifx\svgscale\undefined%
      \relax%
    \else%
      \setlength{\unitlength}{\unitlength * \real{\svgscale}}%
    \fi%
  \else%
    \setlength{\unitlength}{\svgwidth}%
  \fi%
  \global\let\svgwidth\undefined%
  \global\let\svgscale\undefined%
  \makeatother%
  \begin{picture}(1,0.29809052)%
    \put(0,0){\includegraphics[width=\unitlength]{NI.eps}}%
    \put(0.40964611,0.00571559){\color[rgb]{0,0,0}\makebox(0,0)[lb]{\smash{$d_i-20/d_i$}}}%
    \put(0.56219017,0.0067086){\color[rgb]{0,0,0}\makebox(0,0)[lb]{\smash{$d_i-10/d_i$}}}%
    \put(0.81333291,0.17063982){\color[rgb]{0,0,0}\makebox(0,0)[lb]{\smash{$d_iy_i$}}}%
    \put(0.95310295,0.12611926){\color[rgb]{0,0,0}\makebox(0,0)[lb]{\smash{$y_i$}}}%
    \put(0.56402544,0.25249485){\color[rgb]{0,0,0}\makebox(0,0)[lb]{\smash{$p_i$}}}%
    \put(0.27235155,0.11507136){\color[rgb]{0,0,0}\makebox(0,0)[lb]{\smash{$\rho y_i$}}}%
    \put(0.12853069,0.20485954){\color[rgb]{0,0,0}\makebox(0,0)[lb]{\smash{$10\sqrt{1-\rho^2 y_i^2}$}}}%
  \end{picture}%
\endgroup%

%% file: bowl.eps_tex
%% Creator: Inkscape inkscape 0.48.4, www.inkscape.org
%% PDF/EPS/PS + LaTeX output extension by Johan Engelen, 2010
%% Accompanies image file 'bowl.eps' (pdf, eps, ps)
%%
%% To include the image in your LaTeX document, write
%%   \input{<filename>.pdf_tex}
%%  instead of
%%   \includegraphics{<filename>.pdf}
%% To scale the image, write
%%   \def\svgwidth{<desired width>}
%%   \input{<filename>.pdf_tex}
%%  instead of
%%   \includegraphics[width=<desired width>]{<filename>.pdf}
%%
%% Images with a different path to the parent latex file can
%% be accessed with the `import' package (which may need to be
%% installed) using
%%   \usepackage{import}
%% in the preamble, and then including the image with
%%   \import{<path to file>}{<filename>.pdf_tex}
%% Alternatively, one can specify
%%   \graphicspath{{<path to file>/}}
%% 
%% For more information, please see info/svg-inkscape on CTAN:
%%   http://tug.ctan.org/tex-archive/info/svg-inkscape
%%
\begingroup%
  \makeatletter%
  \providecommand\color[2][]{%
    \errmessage{(Inkscape) Color is used for the text in Inkscape, but the package 'color.sty' is not loaded}%
    \renewcommand\color[2][]{}%
  }%
  \providecommand\transparent[1]{%
    \errmessage{(Inkscape) Transparency is used (non-zero) for the text in Inkscape, but the package 'transparent.sty' is not loaded}%
    \renewcommand\transparent[1]{}%
  }%
  \providecommand\rotatebox[2]{#2}%
  \ifx\svgwidth\undefined%
    \setlength{\unitlength}{341.836584bp}%
    \ifx\svgscale\undefined%
      \relax%
    \else%
      \setlength{\unitlength}{\unitlength * \real{\svgscale}}%
    \fi%
  \else%
    \setlength{\unitlength}{\svgwidth}%
  \fi%
  \global\let\svgwidth\undefined%
  \global\let\svgscale\undefined%
  \makeatother%
  \begin{picture}(1,0.66207207)%
    \put(0,0){\includegraphics[width=\unitlength]{bowl.eps}}%
    \put(0.57782873,0.30037225){\color[rgb]{0,0,0}\makebox(0,0)[lb]{\smash{$u(r)$}}}%
    \put(0.56944484,0.59079865){\color[rgb]{0,0,0}\makebox(0,0)[lb]{\smash{$\Sigma_y$}}}%
    \put(0.70890293,0.58897315){\color[rgb]{0,0,0}\makebox(0,0)[lb]{\smash{$2\Sigma_y+\tau$}}}%
    \put(0.45849838,0.01314204){\color[rgb]{0,0,0}\makebox(0,0)[lb]{\smash{$\tau y$}}}%
    \put(0.47470099,0.37987901){\color[rgb]{0,0,0}\makebox(0,0)[lb]{\smash{$r$}}}%
    \put(0.65553141,0.16629539){\color[rgb]{0,0,0}\makebox(0,0)[lb]{\smash{$p$}}}%
    \put(0.45389398,0.61179792){\color[rgb]{0,0,0}\makebox(0,0)[lb]{\smash{$-y$}}}%
    \put(0.42849008,0.18208527){\color[rgb]{0,0,0}\makebox(0,0)[lb]{\smash{$0$}}}%
  \end{picture}%
\endgroup%

%% file: OverP.eps_tex
%% Creator: Inkscape inkscape 0.48.4, www.inkscape.org
%% PDF/EPS/PS + LaTeX output extension by Johan Engelen, 2010
%% Accompanies image file 'OverP.eps' (pdf, eps, ps)
%%
%% To include the image in your LaTeX document, write
%%   \input{<filename>.pdf_tex}
%%  instead of
%%   \includegraphics{<filename>.pdf}
%% To scale the image, write
%%   \def\svgwidth{<desired width>}
%%   \input{<filename>.pdf_tex}
%%  instead of
%%   \includegraphics[width=<desired width>]{<filename>.pdf}
%%
%% Images with a different path to the parent latex file can
%% be accessed with the `import' package (which may need to be
%% installed) using
%%   \usepackage{import}
%% in the preamble, and then including the image with
%%   \import{<path to file>}{<filename>.pdf_tex}
%% Alternatively, one can specify
%%   \graphicspath{{<path to file>/}}
%% 
%% For more information, please see info/svg-inkscape on CTAN:
%%   http://tug.ctan.org/tex-archive/info/svg-inkscape
%%
\begingroup%
  \makeatletter%
  \providecommand\color[2][]{%
    \errmessage{(Inkscape) Color is used for the text in Inkscape, but the package 'color.sty' is not loaded}%
    \renewcommand\color[2][]{}%
  }%
  \providecommand\transparent[1]{%
    \errmessage{(Inkscape) Transparency is used (non-zero) for the text in Inkscape, but the package 'transparent.sty' is not loaded}%
    \renewcommand\transparent[1]{}%
  }%
  \providecommand\rotatebox[2]{#2}%
  \ifx\svgwidth\undefined%
    \setlength{\unitlength}{332.75bp}%
    \ifx\svgscale\undefined%
      \relax%
    \else%
      \setlength{\unitlength}{\unitlength * \real{\svgscale}}%
    \fi%
  \else%
    \setlength{\unitlength}{\svgwidth}%
  \fi%
  \global\let\svgwidth\undefined%
  \global\let\svgscale\undefined%
  \makeatother%
  \begin{picture}(1,0.45395898)%
    \put(0,0){\includegraphics[width=\unitlength]{OverP.eps}}%
    \put(0.68975399,0.04388969){\color[rgb]{0,0,0}\makebox(0,0)[lb]{\smash{$P$}}}%
    \put(0.13250619,0.37113076){\color[rgb]{0,0,0}\makebox(0,0)[lb]{\smash{$\rho_2y$}}}%
    \put(0.58190789,0.3532346){\color[rgb]{0,0,0}\makebox(0,0)[lb]{\smash{$\rho_2\gamma_k$}}}%
    \put(0.66348869,0.40891097){\color[rgb]{0,0,0}\makebox(0,0)[lb]{\smash{$\rho_2y_k$}}}%
    \put(0.3711849,0.42807154){\color[rgb]{0,0,0}\makebox(0,0)[lb]{\smash{$\mathbb R_y$}}}%
    \put(0.80492106,0.18738075){\color[rgb]{0,0,0}\makebox(0,0)[lb]{\smash{$z_1$}}}%
  \end{picture}%
\endgroup%

%% file: OverP1.eps_tex
%% Creator: Inkscape inkscape 0.48.4, www.inkscape.org
%% PDF/EPS/PS + LaTeX output extension by Johan Engelen, 2010
%% Accompanies image file 'OverP1.eps' (pdf, eps, ps)
%%
%% To include the image in your LaTeX document, write
%%   \input{<filename>.pdf_tex}
%%  instead of
%%   \includegraphics{<filename>.pdf}
%% To scale the image, write
%%   \def\svgwidth{<desired width>}
%%   \input{<filename>.pdf_tex}
%%  instead of
%%   \includegraphics[width=<desired width>]{<filename>.pdf}
%%
%% Images with a different path to the parent latex file can
%% be accessed with the `import' package (which may need to be
%% installed) using
%%   \usepackage{import}
%% in the preamble, and then including the image with
%%   \import{<path to file>}{<filename>.pdf_tex}
%% Alternatively, one can specify
%%   \graphicspath{{<path to file>/}}
%% 
%% For more information, please see info/svg-inkscape on CTAN:
%%   http://tug.ctan.org/tex-archive/info/svg-inkscape
%%
\begingroup%
  \makeatletter%
  \providecommand\color[2][]{%
    \errmessage{(Inkscape) Color is used for the text in Inkscape, but the package 'color.sty' is not loaded}%
    \renewcommand\color[2][]{}%
  }%
  \providecommand\transparent[1]{%
    \errmessage{(Inkscape) Transparency is used (non-zero) for the text in Inkscape, but the package 'transparent.sty' is not loaded}%
    \renewcommand\transparent[1]{}%
  }%
  \providecommand\rotatebox[2]{#2}%
  \ifx\svgwidth\undefined%
    \setlength{\unitlength}{334.6bp}%
    \ifx\svgscale\undefined%
      \relax%
    \else%
      \setlength{\unitlength}{\unitlength * \real{\svgscale}}%
    \fi%
  \else%
    \setlength{\unitlength}{\svgwidth}%
  \fi%
  \global\let\svgwidth\undefined%
  \global\let\svgscale\undefined%
  \makeatother%
  \begin{picture}(1,0.537637)%
    \put(0,0){\includegraphics[width=\unitlength]{OverP1.eps}}%
    \put(0.83392027,0.04561918){\color[rgb]{0,0,0}\makebox(0,0)[lb]{\smash{$P_1$}}}%
    \put(0.308656,0.48700265){\color[rgb]{0,0,0}\makebox(0,0)[lb]{\smash{$\mathbb R_y$}}}%
    \put(0.3735998,0.19041728){\color[rgb]{0,0,0}\makebox(0,0)[lb]{\smash{$x_1$}}}%
    \put(0.4902575,0.20951337){\color[rgb]{0,0,0}\makebox(0,0)[lb]{\smash{$x_2$}}}%
    \put(0.34789492,0.39050621){\color[rgb]{0,0,0}\makebox(0,0)[lb]{\smash{$\rho_3 x_1$}}}%
    \put(0.79653134,0.42177694){\color[rgb]{0,0,0}\makebox(0,0)[lb]{\smash{$\rho_3 x_2$}}}%
    \put(0.6497983,0.10089108){\color[rgb]{0,0,0}\makebox(0,0)[lb]{\smash{$\rho_2\gamma_k$}}}%
    \put(0.640798,0.34062097){\color[rgb]{0,0,0}\makebox(0,0)[lb]{\smash{$z_2$}}}%
  \end{picture}%
\endgroup%

%% file: one_end.eps_tex
%% Creator: Inkscape inkscape 0.48.4, www.inkscape.org
%% PDF/EPS/PS + LaTeX output extension by Johan Engelen, 2010
%% Accompanies image file 'one_end.eps' (pdf, eps, ps)
%%
%% To include the image in your LaTeX document, write
%%   \input{<filename>.pdf_tex}
%%  instead of
%%   \includegraphics{<filename>.pdf}
%% To scale the image, write
%%   \def\svgwidth{<desired width>}
%%   \input{<filename>.pdf_tex}
%%  instead of
%%   \includegraphics[width=<desired width>]{<filename>.pdf}
%%
%% Images with a different path to the parent latex file can
%% be accessed with the `import' package (which may need to be
%% installed) using
%%   \usepackage{import}
%% in the preamble, and then including the image with
%%   \import{<path to file>}{<filename>.pdf_tex}
%% Alternatively, one can specify
%%   \graphicspath{{<path to file>/}}
%% 
%% For more information, please see info/svg-inkscape on CTAN:
%%   http://tug.ctan.org/tex-archive/info/svg-inkscape
%%
\begingroup%
  \makeatletter%
  \providecommand\color[2][]{%
    \errmessage{(Inkscape) Color is used for the text in Inkscape, but the package 'color.sty' is not loaded}%
    \renewcommand\color[2][]{}%
  }%
  \providecommand\transparent[1]{%
    \errmessage{(Inkscape) Transparency is used (non-zero) for the text in Inkscape, but the package 'transparent.sty' is not loaded}%
    \renewcommand\transparent[1]{}%
  }%
  \providecommand\rotatebox[2]{#2}%
  \ifx\svgwidth\undefined%
    \setlength{\unitlength}{312.65bp}%
    \ifx\svgscale\undefined%
      \relax%
    \else%
      \setlength{\unitlength}{\unitlength * \real{\svgscale}}%
    \fi%
  \else%
    \setlength{\unitlength}{\svgwidth}%
  \fi%
  \global\let\svgwidth\undefined%
  \global\let\svgscale\undefined%
  \makeatother%
  \begin{picture}(1,0.71632422)%
    \put(0,0){\includegraphics[width=\unitlength]{one_end.eps}}%
    \put(0.75359898,0.25155535){\color[rgb]{0,0,0}\makebox(0,0)[lb]{\smash{$p_i$}}}%
    \put(0.02539871,0.41276125){\color[rgb]{0,0,0}\makebox(0,0)[lb]{\smash{$ky=x_1^2+x_2^2$}}}%
    \put(0.62672308,0.63818055){\color[rgb]{0,0,0}\makebox(0,0)[lb]{\smash{$M$}}}%
    \put(0.82516314,0.43275931){\color[rgb]{0,0,0}\makebox(0,0)[lb]{\smash{$p_i+s_{i,k}y$}}}%
    \put(0.17706973,0.12424318){\color[rgb]{0,0,0}\makebox(0,0)[lb]{\smash{$-ky=x_1^2+x_2^2$}}}%
    \put(0.85685886,0.59595988){\color[rgb]{0,0,0}\makebox(0,0)[lb]{\smash{$2ky=x_1^2+x_2^2$}}}%
    \put(0.43452182,0.5327485){\color[rgb]{0,0,0}\makebox(0,0)[lb]{\smash{$r_iy$}}}%
  \end{picture}%
\endgroup%

%% file: compactness_shrinker_R_v3.bbl
% \bib, bibdiv, biblist are defined by the amsrefs package.
\begin{bibdiv}
\begin{biblist}

\bib{ADS19}{article}{
      author={Angenent, Sigurd},
      author={Daskalopoulos, Panagiota},
      author={Sesum, Natasa},
       title={Unique asymptotics of ancient convex mean curvature flow
  solutions},
        date={2019},
        ISSN={0022-040X},
     journal={J. Differential Geom.},
      volume={111},
      number={3},
       pages={381\ndash 455},
         url={https://doi.org/10.4310/jdg/1552442605},
      review={\MR{3934596}},
}

\bib{Ang92}{incollection}{
      author={Angenent, Sigurd~B.},
       title={Shrinking doughnuts},
        date={1992},
   booktitle={Nonlinear diffusion equations and their equilibrium states, 3
  ({G}regynog, 1989)},
      series={Progr. Nonlinear Differential Equations Appl.},
      volume={7},
   publisher={Birkh\"{a}user Boston, Boston, MA},
       pages={21\ndash 38},
      review={\MR{1167827}},
}

\bib{BK}{article}{
      author={Bauer, Matthias},
       title={Existence of minimizing {W}illmore surfaces of prescribed genus},
        date={2003},
        ISSN={1073-7928},
     journal={Int. Math. Res. Not.},
      number={10},
       pages={553\ndash 576},
      review={\MR{1941840}},
      review={\MR{1941840}},
}

\bib{BW16}{article}{
      author={Bernstein, Jacob},
      author={Wang, Lu},
       title={A sharp lower bound for the entropy of closed hypersurfaces up to
  dimension six},
        date={2016},
        ISSN={0020-9910},
     journal={Invent. Math.},
      volume={206},
      number={3},
       pages={601\ndash 627},
         url={https://doi.org/10.1007/s00222-016-0659-3},
      review={\MR{3573969}},
}

\bib{BW17}{article}{
      author={Bernstein, Jacob},
      author={Wang, Lu},
       title={A topological property of asymptotically conical self-shrinkers
  of small entropy},
        date={2017},
        ISSN={0012-7094},
     journal={Duke Math. J.},
      volume={166},
      number={3},
       pages={403\ndash 435},
         url={https://doi.org/10.1215/00127094-3715082},
      review={\MR{3606722}},
}

\bib{BW18}{article}{
      author={Bernstein, Jacob},
      author={Wang, Lu},
       title={Hausdorff stability of the round two-sphere under small
  perturbations of the entropy},
        date={2018},
        ISSN={1073-2780},
     journal={Math. Res. Lett.},
      volume={25},
      number={2},
       pages={347\ndash 365},
         url={https://doi.org/10.4310/MRL.2018.v25.n2.a1},
      review={\MR{3826825}},
}

\bib{Bra78}{book}{
      author={Brakke, Kenneth~A.},
       title={The motion of a surface by its mean curvature},
      series={Mathematical Notes},
   publisher={Princeton University Press, Princeton, N.J.},
        date={1978},
      volume={20},
        ISBN={0-691-08204-9},
      review={\MR{485012}},
}

\bib{Bre16}{article}{
      author={Brendle, Simon},
       title={Embedded self-similar shrinkers of genus 0},
        date={2016},
        ISSN={0003-486X},
     journal={Ann. of Math. (2)},
      volume={183},
      number={2},
       pages={715\ndash 728},
         url={https://doi.org/10.4007/annals.2016.183.2.6},
      review={\MR{3450486}},
}

\bib{BC17}{article}{
      author={Brendle, Simon},
      author={Choi, Kyeongsu},
       title={Uniqueness of convex ancient solutions to mean curvature flow in
  {$\Bbb R^3$}},
        date={2019},
        ISSN={0020-9910},
     journal={Invent. Math.},
      volume={217},
      number={1},
       pages={35\ndash 76},
         url={https://doi.org/10.1007/s00222-019-00859-4},
      review={\MR{3958790}},
}

\bib{CMZ15}{article}{
      author={Cheng, Xu},
      author={Mejia, Tito},
      author={Zhou, Detang},
       title={Stability and compactness for complete {$f$}-minimal surfaces},
        date={2015},
        ISSN={0002-9947},
     journal={Trans. Amer. Math. Soc.},
      volume={367},
      number={6},
       pages={4041\ndash 4059},
         url={https://doi.org/10.1090/S0002-9947-2015-06207-2},
      review={\MR{3324919}},
}

\bib{CZ13}{article}{
      author={Cheng, Xu},
      author={Zhou, Detang},
       title={Volume estimate about shrinkers},
        date={2013},
        ISSN={0002-9939},
     journal={Proc. Amer. Math. Soc.},
      volume={141},
      number={2},
       pages={687\ndash 696},
         url={https://doi.org/10.1090/S0002-9939-2012-11922-7},
      review={\MR{2996973}},
}

\bib{CS85}{article}{
      author={Choi, Hyeong~In},
      author={Schoen, Richard},
       title={The space of minimal embeddings of a surface into a
  three-dimensional manifold of positive {R}icci curvature},
        date={1985},
        ISSN={0020-9910},
     journal={Invent. Math.},
      volume={81},
      number={3},
       pages={387\ndash 394},
         url={https://doi.org/10.1007/BF01388577},
      review={\MR{807063}},
}

\bib{CHH18}{article}{
      author={{Choi}, Kyeongsu},
      author={{Haslhofer}, Robert},
      author={{Hershkovits}, O},
       title={{Ancient low entropy flows, mean convex neighborhoods, and
  uniqueness}},
        date={2018-10},
     journal={ArXiv e-prints},
      eprint={1810.08467},
}

\bib{CM-minimal1}{article}{
      author={Colding, Tobias~H.},
      author={Minicozzi, William~P., II},
       title={The space of embedded minimal surfaces of fixed genus in a
  3-manifold. {I}. {E}stimates off the axis for disks},
        date={2004},
        ISSN={0003-486X},
     journal={Ann. of Math. (2)},
      volume={160},
      number={1},
       pages={27\ndash 68},
         url={https://doi.org/10.4007/annals.2004.160.27},
      review={\MR{2119717}},
}

\bib{CM-minimal2}{article}{
      author={Colding, Tobias~H.},
      author={Minicozzi, William~P., II},
       title={The space of embedded minimal surfaces of fixed genus in a
  3-manifold. {II}. {M}ulti-valued graphs in disks},
        date={2004},
        ISSN={0003-486X},
     journal={Ann. of Math. (2)},
      volume={160},
      number={1},
       pages={69\ndash 92},
         url={https://doi.org/10.4007/annals.2004.160.69},
      review={\MR{2119718}},
}

\bib{CM-minimal3}{article}{
      author={Colding, Tobias~H.},
      author={Minicozzi, William~P., II},
       title={The space of embedded minimal surfaces of fixed genus in a
  3-manifold. {III}. {P}lanar domains},
        date={2004},
        ISSN={0003-486X},
     journal={Ann. of Math. (2)},
      volume={160},
      number={2},
       pages={523\ndash 572},
         url={https://doi.org/10.4007/annals.2004.160.523},
      review={\MR{2123932}},
}

\bib{CM-minimal4}{article}{
      author={Colding, Tobias~H.},
      author={Minicozzi, William~P., II},
       title={The space of embedded minimal surfaces of fixed genus in a
  3-manifold. {IV}. {L}ocally simply connected},
        date={2004},
        ISSN={0003-486X},
     journal={Ann. of Math. (2)},
      volume={160},
      number={2},
       pages={573\ndash 615},
         url={https://doi.org/10.4007/annals.2004.160.573},
      review={\MR{2123933}},
}

\bib{CM12_1}{article}{
      author={Colding, Tobias~H.},
      author={Minicozzi, William~P., II},
       title={Generic mean curvature flow {I}: generic singularities},
        date={2012},
        ISSN={0003-486X},
     journal={Ann. of Math. (2)},
      volume={175},
      number={2},
       pages={755\ndash 833},
         url={https://doi.org/10.4007/annals.2012.175.2.7},
      review={\MR{2993752}},
}

\bib{CM12_2}{article}{
      author={Colding, Tobias~H.},
      author={Minicozzi, William~P., II},
       title={Smooth compactness of self-shrinkers},
        date={2012},
        ISSN={0010-2571},
     journal={Comment. Math. Helv.},
      volume={87},
      number={2},
       pages={463\ndash 475},
         url={https://doi.org/10.4171/CMH/260},
      review={\MR{2914856}},
}

\bib{CM-minimal5}{article}{
      author={Colding, Tobias~H.},
      author={Minicozzi, William~P., II},
       title={The space of embedded minimal surfaces of fixed genus in a
  3-manifold {V}; fixed genus},
        date={2015},
        ISSN={0003-486X},
     journal={Ann. of Math. (2)},
      volume={181},
      number={1},
       pages={1\ndash 153},
         url={https://doi.org/10.4007/annals.2015.181.1.1},
      review={\MR{3272923}},
}

\bib{CIMW}{article}{
      author={Colding, Tobias~Holck},
      author={Ilmanen, Tom},
      author={Minicozzi, William~P., II},
      author={White, Brian},
       title={The round sphere minimizes entropy among closed self-shrinkers},
        date={2013},
        ISSN={0022-040X},
     journal={J. Differential Geom.},
      volume={95},
      number={1},
       pages={53\ndash 69},
         url={http://projecteuclid.org/euclid.jdg/1375124609},
      review={\MR{3128979}},
}

\bib{GH86}{article}{
      author={Gage, M.},
       title={The heat equation shrinking convex plane curves},
        date={1986},
        ISSN={0022-040X},
     journal={J. Differential Geom.},
      volume={23},
      number={1},
       pages={69\ndash 96},
         url={http://projecteuclid.org/euclid.jdg/1214439902},
      review={\MR{840401}},
}

\bib{Grayson}{article}{
      author={Grayson, Matthew~A.},
       title={The heat equation shrinks embedded plane curves to round points},
        date={1987},
        ISSN={0022-040X},
     journal={J. Differential Geom.},
      volume={26},
      number={2},
       pages={285\ndash 314},
         url={http://projecteuclid.org/euclid.jdg/1214441371},
      review={\MR{906392}},
}

\bib{Gua16}{article}{
      author={Guang, Qiang},
       title={Volume growth, entropy and stability for translating solitons},
        date={2019},
        ISSN={1019-8385},
     journal={Comm. Anal. Geom.},
      volume={27},
      number={1},
       pages={47\ndash 72},
         url={https://doi.org/10.4310/CAG.2019.v27.n1.a2},
      review={\MR{3951020}},
}

\bib{Her18}{article}{
      author={{Hershkovits}, Or},
       title={{Translators asymptotic to cylinders}},
     journal={to appear in Crelle's Journal},
      eprint={1805.10553},
}

\bib{Hui90}{article}{
      author={Huisken, Gerhard},
       title={Asymptotic behavior for singularities of the mean curvature
  flow},
        date={1990},
        ISSN={0022-040X},
     journal={J. Differential Geom.},
      volume={31},
      number={1},
       pages={285\ndash 299},
         url={http://projecteuclid.org/euclid.jdg/1214444099},
      review={\MR{1030675}},
}

\bib{Ilm94}{article}{
      author={Ilmanen, Tom},
       title={Elliptic regularization and partial regularity for motion by mean
  curvature},
        date={1994},
        ISSN={0065-9266},
     journal={Mem. Amer. Math. Soc.},
      volume={108},
      number={520},
       pages={x+90},
         url={https://doi.org/10.1090/memo/0520},
      review={\MR{1196160}},
}

\bib{Ilm95}{article}{
      author={Ilmanen, Tom},
       title={Singularities of mean curvature flow of surfaces},
        date={1995},
     journal={preprint},
}

\bib{KKM18}{article}{
      author={Kapouleas, Nikolaos},
      author={Kleene, Stephen~James},
      author={M{\o}ller, Niels~Martin},
       title={Mean curvature self-shrinkers of high genus: non-compact
  examples},
        date={2018},
        ISSN={0075-4102},
     journal={J. Reine Angew. Math.},
      volume={739},
       pages={1\ndash 39},
         url={https://doi.org/10.1515/crelle-2015-0050},
      review={\MR{3808256}},
}

\bib{KZ}{article}{
      author={Ketover, Daniel},
      author={Zhou, Xin},
       title={Entropy of closed surfaces and min-max theory},
        date={2018},
        ISSN={0022-040X},
     journal={J. Differential Geom.},
      volume={110},
      number={1},
       pages={31\ndash 71},
         url={https://doi.org/10.4310/jdg/1536285626},
      review={\MR{3851744}},
}

\bib{MHS15}{article}{
      author={Mart{\'\i}n, Francisco},
      author={Savas-Halilaj, Andreas},
      author={Smoczyk, Knut},
       title={On the topology of translating solitons of the mean curvature
  flow},
        date={2015},
        ISSN={0944-2669},
     journal={Calc. Var. Partial Differential Equations},
      volume={54},
      number={3},
       pages={2853\ndash 2882},
         url={https://doi.org/10.1007/s00526-015-0886-2},
      review={\MR{3412395}},
}

\bib{MeeksRosenberg05}{article}{
      author={Meeks, William~H., III},
      author={Rosenberg, Harold},
       title={The uniqueness of the helicoid},
        date={2005},
        ISSN={0003-486X},
     journal={Ann. of Math. (2)},
      volume={161},
      number={2},
       pages={727\ndash 758},
         url={https://doi.org/10.4007/annals.2005.161.727},
      review={\MR{2153399}},
}

\bib{Nguyen}{article}{
      author={Nguyen, Xuan~Hien},
       title={Construction of complete embedded self-similar surfaces under
  mean curvature flow, {P}art {III}},
        date={2014},
        ISSN={0012-7094},
     journal={Duke Math. J.},
      volume={163},
      number={11},
       pages={2023\ndash 2056},
         url={https://doi.org/10.1215/00127094-2795108},
      review={\MR{3263027}},
}

\bib{Sim93}{article}{
      author={Simon, Leon},
       title={Existence of surfaces minimizing the {W}illmore functional},
        date={1993},
        ISSN={1019-8385},
     journal={Comm. Anal. Geom.},
      volume={1},
      number={2},
       pages={281\ndash 326},
         url={https://doi.org/10.4310/CAG.1993.v1.n2.a4},
      review={\MR{1243525}},
}

\bib{Song14}{article}{
      author={{Song}, Antoine},
       title={{A maximum principle for self-shrinkers and some consequences}},
        date={2014-12},
     journal={ArXiv e-prints},
      eprint={1412.4755},
}

\bib{Sun18}{article}{
      author={{Sun}, Ao},
       title={{Local Entropy and Generic Multiplicity One Singularities of Mean
  Curvature Flow of Surfaces}},
        date={2018-10},
     journal={ArXiv e-prints},
      eprint={1810.08114},
}

\bib{Wang14}{article}{
      author={Wang, Lu},
       title={Uniqueness of self-similar shrinkers with asymptotically conical
  ends},
        date={2014},
        ISSN={0894-0347},
     journal={J. Amer. Math. Soc.},
      volume={27},
      number={3},
       pages={613\ndash 638},
      review={\MR{3194490}},
}

\bib{Wang16}{article}{
      author={Wang, Lu},
       title={Asymptotic structure of self-shrinkers},
        date={2016-10},
     journal={ArXiv e-prints},
      eprint={1610.04904},
}

\bib{Wang16_2}{article}{
      author={Wang, Lu},
       title={Uniqueness of self-similar shrinkers with asymptotically
  cylindrical ends},
        date={2016},
        ISSN={0075-4102},
     journal={J. Reine Angew. Math.},
      volume={715},
       pages={207\ndash 230},
         url={https://doi.org/10.1515/crelle-2014-0006},
      review={\MR{3507924}},
}

\bib{Whi94}{article}{
      author={White, Brian},
       title={Partial regularity of mean-convex hypersurfaces flowing by mean
  curvature},
        date={1994},
     journal={Internat. Math. Res. Notices},
      number={4},
       pages={1073\ndash 7928},
         url={https://doi.org/10.1155/S1073792894000206},
      review={\MR{1266114}},
}

\bib{Whi02}{article}{
      author={White, Brian},
       title={A local regularity theorem for mean curvature flow},
        date={2005},
        ISSN={0003-486X},
     journal={Ann. of Math. (2)},
      volume={161},
      number={3},
       pages={1487\ndash 1519},
         url={https://doi.org/10.4007/annals.2005.161.1487},
      review={\MR{2180405}},
}

\bib{Whi09}{article}{
      author={White, Brian},
       title={Currents and flat chains associated to varifolds, with an
  application to mean curvature flow},
        date={2009},
        ISSN={0012-7094},
     journal={Duke Math. J.},
      volume={148},
      number={1},
       pages={41\ndash 62},
         url={https://doi.org/10.1215/00127094-2009-019},
      review={\MR{2515099}},
}

\bib{Zhu16}{article}{
      author={{Zhu}, Jonathan~J.},
       title={{On the entropy of closed hypersurfaces and singular
  self-shrinkers}},
     journal={to appear in JDG},
      eprint={1607.07760},
}

\end{biblist}
\end{bibdiv}
